\documentclass[oneside]{amsart}

\usepackage{geometry}                
\usepackage{graphicx}
\usepackage{amssymb}
\usepackage{amsmath}
\usepackage{amsthm}                
\usepackage{mathrsfs}  
\usepackage{epstopdf}
\usepackage[makeroom]{cancel}
\usepackage{stmaryrd}
\usepackage{xargs}
\usepackage{mathdots}
\usepackage{mathtools}
\usepackage{tikz-cd}
\usepackage{relsize}
\usepackage{scalerel}
\usetikzlibrary{matrix}
\usetikzlibrary{arrows.meta}
\usepackage{changepage}
\usepackage{hyperref}
\usepackage[capitalise]{cleveref}
\usepackage{tikz}
\usetikzlibrary{calc}
\usepackage[shortlabels]{enumitem}

\newcommand{\bctikz}{$$\begin{tikzcd}}
\newcommand{\ectikz}{\end{tikzcd}$$}
\newcommand{\ntikz}{\end{tikzcd}\qquad\qquad\begin{tikzcd}}

\newtheorem{thm}{Theorem}[section]
\newtheorem*{thm*}{Theorem}
\newtheorem{lem}[thm]{Lemma}
\newtheorem{prop}[thm]{Proposition}
\newtheorem{cor}[thm]{Corollary}

\theoremstyle{definition}
\newtheorem{defn}[thm]{Definition}
\theoremstyle{definition}
\newtheorem{ex}[thm]{Example}
\theoremstyle{remark}

\newtheorem{rem}[thm]{Remark}

\newcommand{\ul}{\underline}
\newcommand{\cat}[1]{{\mathbf{#1}}} 
\newcommand{\smallcat}[1]{{\mathcal{#1}}} 
\newcommand{\bicat}{\cat} 
\newcommand{\cell}[1]{{\widehat{#1}}} 

\newcommand{\blp}{\Big{(}}
\newcommand{\brp}{\Big{)}}

\newcommand{\emb}{\hookrightarrow}
\newcommand{\atol}{\xrightarrow}

\newcommand{\adj}{\rightleftharpoons}
\newcommand{\fib}{\twoheadrightarrow}

\newcommand{\scdots}[2][]{\mathinner{#1\overset{#2}{\cdots}#1}}

\DeclareMathOperator{\Hom}{Hom}

\DeclareMathOperator{\id}{id}
\DeclareMathOperator{\Ob}{Ob}

\DeclareMathOperator*{\colim}{colim}
\newcommand{\coprodl}{\coprod\limits}
\newcommand{\bigint}{\oint}
\newcommand{\sint}{\smallint}
\newcommand{\soint}{{\textstyle\oint}}

\newcommand{\ints}{\mathbb{Z}}
\newcommand{\nats}{\mathbb{N}}
\newcommand{\Set}{\cat{Set}}

\newcommand{\Cat}{\cat{Cat}}
\newcommand{\CAT}{\cat{CAT}}

\newcommand{\G}{\cat{G}}

\newcommand{\A}{\smallcat{A}}
\newcommand{\B}{\smallcat{B}}
\newcommand{\C}{\smallcat{C}}

\newcommand{\R}{\smallcat{R}}

\newcommand{\X}{\smallcat{X}}

\newcommand{\ah}{\cell{\A}}
\newcommand{\bh}{\cell{\B}}
\newcommand{\ch}{\cell{\C}}

\newcommand{\rh}{\cell{\R}}
\newcommand{\rhp}{{\cell{\R}_+}}
\newcommand{\chp}{{\ch'}}
\newcommand{\chpp}{{\ch''}}
\newcommand{\chppp}{{\ch'''}}
\newcommand{\sqh}{\cell{\square}}

\newcommand{\delt}{\mathop{\mathchoice{\mathlarger{\mathlarger{\Delta}}}{\mathlarger{\mathlarger{\Delta}}}{\scriptstyle \Delta}{\scriptscriptstyle \Delta}}}
\renewcommand{\prod}{\mathop{\mathchoice{\mathlarger{\mathlarger{\Pi}}}{\mathlarger{\mathlarger{\Pi}}}{\scriptstyle \Pi}{\scriptscriptstyle \Pi}}}
\renewcommand{\sum}{\mathop{\mathchoice{\mathlarger{\mathlarger{\Sigma}}}{\mathlarger{\mathlarger{\Sigma}}}{\scriptstyle \Sigma}{\scriptscriptstyle \Sigma}}}

\newcommand{\starrows}{{\;\mathinner{\underset{\tau}{\overset{\sigma}{\rightrightarrows}}}\;}}
\renewcommand{\d}{\partial}
\newcommand{\eps}{\varepsilon}
\newcommand{\s}{\square}

\newcommand{\Th}{\Theta}

\newcommand{\Thth}{{\cell{\Th}_T}}
\newcommand{\Thoh}{{\cell{\Th}_0}}

\newlength\thetawidth
\newlength\thetaheight
\newcommand*{\STheta}{
  \begingroup
    \setlength{\mathsurround}{0pt}
    \settowidth{\thetawidth}{$\Theta$}
    \settoheight{\thetaheight}{$\Theta$}
    \begin{tikzpicture}[x=\thetawidth, y=\thetaheight,]
    \draw[line width=.1\thetawidth, line cap=round]
      (0.125,.05\thetawidth) -- +($(0,1)-(0,.1\thetawidth)$)
      (0.875,.05\thetawidth) -- +($(0,1)-(0,.1\thetawidth)$);
    \draw[line width=.05\thetaheight, line cap=butt]
      (0.125,.025\thetaheight) -- +(0.75,0)
      ($(0.125,1)-(0,.025\thetaheight)$) -- +(0.75,0);
    \draw[line width=.1\thetawidth, line cap=butt]
      (0.275,0.5) -- (0.725,0.5);
    \draw[line width=.0375\thetaheight, line cap=butt]
      ($(0.275,0.5)-(0,.1\thetawidth)$) -- +(0,.2\thetawidth)
      ($(0.725,0.5)-(0,.1\thetawidth)$) -- +(0,.2\thetawidth);
    \end{tikzpicture}
  \endgroup
}

\newcommand{\Fam}{\bicat{Fam}}
\newcommand{\Rep}{\bicat{Rep}}
\newcommand{\Poly}{\bicat{Poly}}
\renewcommand{\H}{H}
\newcommand{\Hse}{{\H_{(S,E)}}}
\newcommand{\Hsep}{{\H_{(S',E')}}}
\newcommand{\g}{{\mathfrak{gr}}}

\newcommand{\uln}{{\ul{n}}}

\title{Familial Monads as Higher Category Theories}
\author{Brandon T. Shapiro}
\date{}                                           

\begin{document}
\maketitle

\begin{abstract}
Categories, $n$-categories, double categories, and multicategories (among others) all have similar definitions as collections of cells with composition operations. We give an explicit description of the information required to define any higher category structure which arises as algebras for a familially representable monad on a presheaf category, then use this to describe several examples relating to higher category theory and cubical sets. The proof of this characterization avoids tedious naturality arguments by passing through the theory of categorical polynomials; along the way, we give descriptions of pullbacks, composites, and exponentiations of split opfibrations in terms of their classifying functors which may be of independent interest.
\end{abstract}

\setcounter{tocdepth}{1}
\tableofcontents

\section*{Introduction}

Categories, $n$-categories, bicategories, double categories, multicategories, and even most types of monoidal categories all have many features in common. In each structure there are some sort of underlying cells (such as objects and arrows), certain arrangements of those cells which can be composed, and equations between different ways of composing the same arrangement. The structure of a category can be described as its underlying graph equipped with additional algebraic structure: a category is an algebra for the free category monad on graphs. Similarly, all of the structures above are algebras for a \emph{familially representable} (or simply \emph{familial}) monad on a presheaf category $\ch := \Set^{\C^{op}}$.


These structures all admit nerve functors, enrichments, notions of limit and colimit, various notions of weak equivalence, and versions of the Yoneda lemma. It would be reasonable to expect that many of these constructions can be unified using the language of familial monads. Weber's Nerve Theorem (\cite[Theorem 4.10]{WeberPra}) shows precisely this for nerve functors: for every familial\footnote{Weber in fact proves this more generally for any ``monad with arities.''} monad $T$, there is a fully faithful functor from algebras of $T$ to a presheaf category $\Thth$. 

The goal of this program is to generalize features of the theory of categories to more general high- and low-dimensional categorical structures modeled as familial monads. In this first paper, we provide a language for defining familial monads in terms of the ``composition operations'' in their algebras, which we use to straightforwardly construct new examples and, in subsequent work, generalize constructions from category theory. In a follow up paper \cite{enrichment} we define enrichment for algebras of any familial monad 
(for instance, what it means to have multicategory-enriched double categories). 
The following main result provides a concrete description of the data that determines a familial monad, designed for ease of defining new such monads based on categorical structures found in the wild.

\begin{thm*}{\textup{(\cref{monad_representation})}}
For a small category $\C$ of ``cell shapes,'' a familial monad $T$ on the presheaf category $\ch$ is completely specified by the following data:
\begin{itemize}
	\item A functor $S \colon \C^{op} \to \Set$, where $t \in Sc$ is an ``operation outputting a cell of shape $c$''
	\item A functor $E \colon \int S \to \ch$, where the presheaf $Et$ is the ``arity'' of the operation $t$
	\item For each cell shape $c$ in $\C$, a ``unit'' operation $e(c) \in Sc$ along with an isomorphism $Ee(c) \cong y(c)$
	\item For each operation $t \in Sc$ and $f \colon Et \to S$ in $\ch$, a ``composite'' operation $m(t,f)$ along with an isomorphism
		$$Em(t,f) \cong \colim_{x \in Et} Ef(x)$$
\end{itemize}
subject to several equations. The monad $T$ then sends a presheaf $X$ to the presheaf $TX$ with
$$TX_c = \coprod_{t \in Sc} \Hom(Et,X).$$
\end{thm*}

An algebra $A$ for the monad $T$ is a type of higher category: it has cells that form a presheaf on $\C$, like the underlying graph of a category. For each operation $t \in Sc$ it is equipped with a function $\Hom(Et,A) \to A_c$ sending each diagram of shape $Et$ in $A$ to a $c$-shaped cell in $A$, like the operation sending a string of $n$ arrows in a category to its composite. The unit and composite operations $e(c)$ and $m(t,f)$ ensure that, respectively, there is always an operation sending a $c$-cell to itself and applying one operation to the outputs of others is itself an operation. The equations governing a higher category theory can be expressed in this language, like in the theory of categories (among others) where an associativity equation asserts that two different operations with the same arity built out of nested composites are equal. The data $(\C,S,E,e,m)$ is called a \emph{higher category schema}. 

We use this language to provide several examples of familial monads beyond those mentioned above. The algebraic structure of a higher category typically involves a combination of degeneracies, symmetries, and compositions, and we describe examples illustrating each of these. A common theme in these examples is familial monads on \emph{semicubical sets}, presheaves with minimally structured cubical cells in each dimension. These cubical cell shapes admit a rich variety of degeneracies, symmetries, and compositions which illustrate the broad scope of higher category structures defined by familial monads. In particular we define monads for cubical sets over semicubical sets, symmetric cubical sets over cubical sets, and cubical $\omega$-categories over semicubical sets. We also discuss more general monads for adding degeneracies indexed by any Reedy category and adding symmetries to presheaves over $\C$ indexed by a crossed $\C$-group in the sense of \cite[Definition 2.1]{generalizedreedy}.

It has been shown in \cite[Theorem 8.1]{WeberGen} and \cite[Proposition 3.8]{Shapely} that \emph{familial functors} are equivalent to a category of pairs $(S,E)$ as above, which we call \emph{familial representations}. It follows intuitively that familial monads should admit a similar description, with a translation of the unit and multiplication transformations into the language of familial representations. Examples of familial monads are defined in \cite[Section 9]{WeberGen}, \cite[Example 2.14]{WeberPra}, and \cite[Proposition 2.9]{Shapely} in manners similar to the above, but the only general description of the data determining a familial monad on $\ch$ is that of \cite[Section C.3]{leinster} in the case when $\C$ is a discrete category, and a complete list of the coherence equations this data must satisfy does not appear in the literature.

Our characterization provides a general strategy for defining a familial monad on $\ch$ for any small category $\C$ directly in terms of the operations and equations between them in its algebras. It also allows for familial monads to be studied in the unified language of higher category schema, and for broad classes of familial monads to be defined all at once. In \cite{thesis}, we illustrate this by defining, for familial monads $T$ and $T'$ subject to minimal conditions, a new familial monad $T \wr T'$ whose algebras behave like $T$-algebras enriched in $T'$-algebras.

A direct proof of this characterization, showing how to specify $\eta,\mu$ based on a higher category schema $(\C,S,E,e,m)$ and checking that they satisfy the unit and associativity equations, would require several enormous diagrams and scores of tedious naturality proofs. Instead, we provide a more conceptual argument, taking advantage of the machinery of categorical polynomials as developed in \cite{GambinoKock} and \cite{WeberPoly}. In particular, we define a bicategory $\Rep$ with 1-cells familial representations $(S,E)$ and show that it is biequivalent to the 2-category $\Fam$ of presheaf categories, familial functors, and and cartesian natural transformations by passing through categorical polynomials (\cref{is_bicategory_equivalent}). Restricting this biequivalence to formal monads in these bicategories yields an equivalence between higher category schema and familial monads. In order to reach categorical polynomials and analyze their composition, in \cref{grothendiecks} we describe how the pullbacks, composites, and exponentiations of split opfibrations behave with respect to the inverse Grothendieck construction exhibiting an opfibration as a functor to $\Cat$; while these constructions are relatively straightforward and in many cases understood by experts, we believe they provide useful intuition and were unable to find them in existing literature.

This result is similar in many ways to \cite[Theorem 4.28]{aggregation}, which proves that the bicategory $\mathbb{C}\mathbf{at}^\#$ (equivalent to $\Fam$) is equivalent to the bicategory $\mathbb{C}\mathbf{omod}(\mathbf{Poly})$ of comonoids and bicomodules in the category of polynomial functors, which are themselves familial endofunctors on $\Set$. This equivalence allows a familial monad to be defined as an endo-bicomodule on a polynomial comonad on $\Set$, such as in \cite[Example 3.6]{nerves} for the free symmetric monoidal category monad on graphs. Describing familial monads as polynomial bicomodules is often useful for applications (see for instance \cite{aggregation}) and proving formal properties (such as in \cite{nerves}, where Weber's nerve functor for algebras of a familial monad is shown to satisfy a universal property), but ultimately a completely rigorous approach to defining new examples and expressing general constructions requires the concrete data and equations of a higher category schema: what are the operations in a given type of higher category, what are their arities and output shapes, what are their units, and how do they compose with one another. The approach of this paper is to provide a framework for working with such concrete data, and our proof that $\Rep$ is equivalent to $\Fam$ is very different from the proof in \cite[Section 4]{aggregation}.

\subsection*{Organization}

In \cref{background_familial} we describe the 2-category $\Fam$ of presheaf categories, familial functors, and cartesian natural transformations and discuss well-known examples of familial monads.  In \cref{representability} we outline the proof of the biequivalence between familial representations and familial functors, and describe in more detail the resulting characterization of familial monads. In \cref{examples} we use our characterization to define monads adding degeneracies to presheaves over a direct category to form presheaves over a Reedy category, adding symmetries to presheaves over $\C$ to form presheaves over $\C G$ for a crossed $\C$-group $G$, and freely generating cubical $\omega$-categories.

The appendices complete the main proof using the theory of polynomials in $\Cat$. In \cref{grothendiecks} we recall the various notions of fibration in $\Cat$ and provide proofs of the pullback-stability, composability, and exponentiability of split opfibrations directly in terms of their classifying functors to $\Cat$. In \cref{polynomials} we introduce several special classes of polynomials in $\Cat$ and develop their relationship with familial representations and familial functors. In \cref{rep_is_bicategory} we define the bicategory structure $\Rep$ on familial representations, complete the proof that it is biequivalent to $\Fam$, and give the full list of equations satisfied by a formal monad in $\Rep$.

\subsection*{Notation and Terminology}

For a small category $\C$, we write $\ch$ for the category $\Set^{\C^{op}}$ of presheaves over $\C$. For $X$ a presheaf in $\ch$, we write $X_c$ for its set of ``$c$-cells'' and $X_i \colon X_c \to X_{c'}$ for its action on a morphism $i \colon c' \to c$ of $\C$. 

We write $\sint X$ for its category of elements, whose objects are pairs $(c \in \Ob(\C), x \in X_c)$, often abbreviated as simply $x$, and morphisms of the form $i_x \colon X_i(x) \to x$ for $i \colon c' \to c$ in $\C$ and $x \in X_c$.

We write $*$ for the terminal presheaf in $\ch$, and $\{*_c\}$ for its singleton set of $c$-cells for each $c \in \Ob(\C)$.

We say that a natural transformation between functors is \emph{cartesian} if all of its naturality squares are pullbacks.

We write $\nats$ for the set of natural numbers $0,1,2,...$, and $\uln$ for the set $\{1,...,n\}$, where $\ul{0}$ is the empty set.

We denote by $\Set$, $\Cat$, and $\CAT$ the categories of sets, small categories, and locally small categories respectively, where the latter two are often regarded as 2-categories.

%
%

\section{Familial Monads}\label{background_familial}

A functor from a category $\A$ to $\Set$ is representable if it is isomorphic to $\Hom_\A(A,-)$ for some object $A$ in $\A$, called a \emph{representation} of the functor. $\Set \cong \ch$ when $\C$ is the terminal category and an object $A$ in $\A$ is equivalently a functor from the terminal category to $\A$, which suggests how to define representable functors $\A \to \ch$ for general $\C$. Given a functor $S \colon \C \to \A$, we get a functor $\A \to \ch$ given by $c \mapsto \Hom(Sc,-)$. 

In algebra, one often encounters functors to $\Set$ which are not representable but instead disjoint unions of representables, represented by a family of objects instead of just one.

\begin{ex}
The free monoid functor $\Set \to \Set$ sends a set $X$ to $* \sqcup X \sqcup X^2 \sqcup X^3 \sqcup \cdots$. Each $X^n$ is isomorphic to the set of functions $\Hom_\Set(\uln,X)$, so this functor is represented by the family of sets $\{\uln\}_{n \in \nats}$. Each $\uln$ corresponds to the unique $n$-ary operation in a monoid; this operation has \emph{arity} $\uln$.
\end{ex}

In the free monoid example, the representation of the functor consists of the set $\nats$ and a functor $\nats \to \Set$, regarding $\nats$ as a discrete category. In higher dimensional algebra though, we encounter functors into $\ch$ more general than disjoint unions of representables.

\begin{ex}\label{cat_example}
When $\C = \G_1$, the category $0 \starrows 1$, $\ch$ is the category of directed graphs. For a graph $X$, $X_\sigma \colon X_1 \to X_0$ identifies the source vertex of each edge and $X_\tau$ identifies the target. We write $\atol{n}$ for the graph consisting of a single path of length $n$: $n+1$ vertices and $n$ successive edges connecting them, for $n \ge 0$. When $n=0,1$, these ``walking paths'' include the single vertex and the single edge graphs.

The free category functor $\ch \to \ch$ sends $X$ to the graph with the same vertices and an edge for every (finite, directed) path in $X$, including length 0 paths which consist of just a vertex. The set of paths in $X$ of fixed length $n$ is precisely $\Hom(\atol{n},X)$, and the set of all paths in $X$ is then
$$\coprod_{n \in \nats} \Hom(\atol{n},X).$$

However, the free category functor is not a disjoint union of representables, as paths of all lengths have the same original set of vertices as their sources and targets. But the vertex part of the functor is representable, as $X_0 \cong \Hom(\atol{0},X)$. This functor is then a disjoint union of representables only in each type of cell separately, and the structure maps are also representable: the source vertex of a length $n$ path is the first vertex in the path, and the function $\Hom(\atol{n},X) \to \Hom(\atol{0},X)$ identifying this source is represented by the map of graphs from $\atol{0}$ to $\atol{n}$ picking out the first vertex.

The data of the free category functor then amounts to the sets $\{\atol{n}\}_{n \in \nats}$ and $\{\atol{0}\}$ of graphs and the source/target maps from $\atol{0}$ to $\atol{n}$. This can be described as a functor $S \colon \C^{op} \to \Set$ sending $0$ to $\{0\}$ and sending $1$ to $\nats$, along with a functor $E \colon \sint S \to \ch$ sending $n$ to $\atol{n}$ and $\sigma,\tau \colon 0 \to n$ to the inclusions of the source and target vertices in $\atol{n}$.
\end{ex}

Functors $\ch \to \ch$ of this form are called \emph{familially representable}, or just \emph{familial}, and describe a wide variety of freely generated higher category structures, as we discuss below. The functor $S \colon \C^{op} \to \Set$ describes the operations which output each cell type in $\C$, and $E \colon \sint S \to \ch$ identifies the arity of each operation, which for the composition of $n$ arrows in a category is the graph $\atol{n}$. The equations such as unit and associativity laws for categories are then between operations of the same arity: the composite of an arrow with the composite of two more arrows, in either order, agrees with the operation composing three arrows all at once.

Our main result describes how to represent the data of a familial functor equipped with the structure of a cartesian monad, giving a direct and simplified method for defining new types of higher categories that fit this pattern.

\subsection{Familial Functors}

\begin{defn}
For $\C,\C'$ small categories, a \emph{familial representation from $\C'$ to $\C$} is a pair $(S,E)$ where
\begin{itemize}
	\item $S$ is a functor $\C^{op} \to \Set$
	\item $E$ is a functor $\sint S \to \chp$
\end{itemize}
Associated to $(S,E)$ is a functor $\Hse \colon \chp \to \ch$ given by, for $X$ a presheaf over $\C'$,
\begin{itemize}
	\item $\Hse(X)_c = \coprodl_{t \in Sc} \Hom_\chp(Et,X)$ for each $c \in \Ob(\C)$
	\item $\Hse(X)_i \colon \coprodl_{t \in Sc} \Hom_\chp(Et,X) \to \coprodl_{t' \in Sc'} \Hom_\chp(Et',X)$ for $i \colon c' \to c$ is given by
$$E(i_t)^\ast \colon \Hom_\chp(Et,X) \to \Hom_\chp \blp E(Si(t)),X \brp.$$  
\end{itemize}
We say that $t \in Sc$ is an \emph{operation} with \emph{output} $c$ and \emph{arity} $Et$, as every diagram of shape $Et$ in $X$ contributes a $c$-cell to $\Hse(X)$. Note that we treat $S$ as a functor rather than a presheaf over $\C$, as we prefer to think of presheaves as geometric objects while $S$ plays more of a bookkeeping role, tracking the relationships between the various operations. That said, as a presheaf $S$ is isomorphic to $\Hse(*)$.
\end{defn}

\begin{defn}\label{familial_functor}
A functor $T \colon \chp \to \ch$ is \emph{familially representable}, or simply \emph{familial}, if it is naturally isomorphic to a functor of the form $\Hse$, where $(S,E)$ is called a \emph{familial representation of $T$}.  We write $\Fam(\chp,\ch)$ for the category of familial functors from $\chp$ to $\ch$ and cartesian natural transformations between them.
\end{defn}

\begin{rem}
Familial functors into $\Set$ were introduced in \cite{DiersThesis} as ``locally representable functors''. They are first called ``familially representable'' in \cite{JohnsonWalters}. In \cite{leinster}, familially representable functors $\A \to \ch$ are defined using a slight variation on the notion of familial representations: instead of a pair $(S,E \colon \sint S \to \A)$, they are equivalently represented by a functor from $\C^{op}$ into a category of ``families'' of objects in $\A$. In \cite{WeberGen}, functors which admit ``strict generic factorizations'' are (nontrivially) equivalent to familial functors in the setting of presheaf categories, with $(S,E)$ called the ``spectrum'' and ``exponent'' of a ``parametric representation'' of a functor. In \cite{WeberPra}, these functors are called ``parametric right adjoints'' or ``p.r.a. functors.'' Most of this overlapping terminology remains in current use, so we choose ``familial functors'' and ``familial representations'' as the most suitable for our purposes.
\end{rem}

We now describe morphisms of familial representations, which will correspond precisely to cartesian natural transformations of the associated familial functors.

\begin{defn}
For $(S,E),(S',E')$ familial representations from $\C'$ to $\C$, a morphism $\phi \colon (S,E) \to (S',E')$ consists of a morphism $\phi^S \colon S \to S'$ in $\ch$ and a natural isomorphism $\phi^E$
\bctikz \sint S \ar{rr}{\sint \phi^S} \ar{dr}[swap]{E} \ar[Rightarrow, shorten=19, shift right=7]{rr}{\phi^E} & & \sint S' \ar{dl}{E'} \\ & \ch \ectikz
We write $\Rep(\C',\C)$ for the category of familial representations from $\C'$ to $\C$ and morphisms of this form between them.
\end{defn}

\begin{prop}\label{rep_local_equivalence}
For small categories $\C',\C$, the assignment $(S,E) \mapsto \Hse$ extends to an equivalence of categories $H \colon \Rep(\C',\C) \to \Fam(\chp,\ch)$.
\end{prop}

This is proven in \cite[Proposition 3.8]{Shapely} for the equivalent notion of pointwise familial functors, but for clarity we give a proof here as well. The ideas in this proof are closely related to \cite[Theorem 7.6]{WeberGen}.

\begin{proof}
The functor $H$ sends a morphism $\phi \colon (S,E) \to (S',E')$ to the natural transformation given on $c$-cells ($c \in \Ob(\C)$) by
$$\coprod_{t \in Sc} \Hom_\chp(Et,-) \to \coprod_{t' \in S'c} \Hom_\chp(E't',-)$$
mapping $\Hom_\chp(Et,-)$ to $\Hom_\chp(E'\phi^S(t),-)$ by precomposition with the isomorphism 
$$\phi^E_t \colon Et \cong E'\phi^S(t).$$
This assignment is natural in $c$ precisely because $\phi^E$ is natural in $t$.

$H$ is essentially surjective by \cref{familial_functor}, so it remains only to show it is fully faithful. First, we observe that given any cartesian natural transformation $\psi \colon \Hse \to \Hsep$ we have for each presheaf $X$ over $\C'$ the following diagram, natural in $c$, where the vertical maps are given on each component by postcomposition with the unique map $X \to *$, the upper square is a naturality pullback square, and $\psi_c$ is the unique map making the lower square commute.
\bctikz
\coprodl_{t \in Sc} \Hom_\chp(Et,X) \arrow{d}{} \arrow{r}{} \ar[phantom, shift right=3]{dr}[pos=-.2]{\lrcorner} & \coprodl_{t' \in S'c} \Hom_\chp(E't',X) \arrow{d}{} \\
\coprodl_{t \in Sc} \Hom_\chp(Et,*) \arrow{d}[swap]{\cong} \arrow{r}{} & \coprodl_{t' \in S'c} \Hom_\chp(E't',*) \arrow{d}{\cong}\\
Sc \arrow{r}{\psi_c} & S'c
\ectikz
As the composite square is also a pullback, the top map above restricts to an isomorphism $\Hom_\chp(Et,X) \to \Hom_\chp(E'\psi_c(t))$ for each $t \in Sc$, natural in $X$ and $t$ (as an object of $\sint S$). By the Yoneda lemma, this isomorphism must be given by precomposition with an isomorphism $\psi_t \colon E'\psi_c(t) \cong Et$, natural in $t$. We can then define a morphism of representations $H^{-1}(\psi) \colon (S,E) \to (S,E)$ with $H^{-1}(\psi)^S_c = \psi_c$ and $H^{-1}(\psi)^E_t = \psi_t^{-1}$, and it is straightforward to check that the assignments 
$$H \colon \Rep(\C',\C) \blp (S,E),(S',E') \brp \rightleftarrows \Fam(\chp,\ch) \blp \Hse,\Hsep \brp : H^{-1}$$
are inverse to one another, completing the proof that $H$ is fully faithful.
\end{proof}

We now describe how familial functors form a sub-2-category $\Fam$ of $\CAT$. 

\begin{prop}\label{fam_2cat}
Categories of the form $\ch$ for small $\C$, familial functors, and cartesian natural transformations form a 2-category.
\end{prop}

\begin{proof}
It suffices to show that familial functors are closed under identities and composites; the corresponding properties of cartesian transformations follow immediately, noting that familial functors preserve pullbacks (\cite[Theorem 8.1]{WeberGen}).

That the identity functor is familial is a consequence of the Yoneda lemma, as for each $X$ in $\ch$ we have $X_c \cong \Hom(y(c),X)$, for each object $c$ in $\C$. Given two familial functors $F \colon \chp \to \ch$ and $G \colon \chpp \to \chp$ represented by $(S,E)$ and $(S',E')$ respectively, \cite[Propositions 3.11, 3.12]{Shapely} show that 
$$FG(X)_c \cong  \coprodl_{\substack{t \in Sc, \\ f \colon Et \to S'}} \Hom_\chpp \blp \colim_{x \colon y(c') \to Et} E'f(x), X \brp ,$$
which can also be shown directly using basic properties of limits and colimits.
\end{proof}

These representations for identities and composites are discussed in more detail in \cref{rep_is_bicategory}, where they are shown to form the identities and composites of the bicategory $\Rep$ whose morphism categories are given by $\Rep(\C',\C)$.

\begin{defn}\label{identity_familial}
The familial representation of $\id_\ch$ is given by $(S^0,E^0)$, where $S^0_c = \{*_c\}$ be the terminal functor and $E^0 \colon \sint S^0 \cong \C \atol{y} \ch$.
\end{defn}

\begin{defn}\label{familial_composite}
For $F,G$ familial functors as above represented by $(S,E),(S',E')$, the familial representation of $FG$ is given by $(SS',EE')$, where 
$$SS'_c = \coprod_{t \in Sc} \Hom_\chp(Et,S) \qquad \textrm{and} \qquad EE'(t,f) = \colim_{x \colon y(c') \to Et} E'f(x).$$
\end{defn}

\subsection{Examples of familial monads}

\begin{defn}
A \emph{familial monad} is a monad on a presheaf category $\ch$ whose functor part is familial and whose unit and multiplication transformations are cartesian.
\end{defn}

A familial monad is the same as a formal monad in the 2-category $\Fam$, and this description will facilitate our characterization of familial monads in \cref{monad_representation}.

Familial monads are of interest for their precision in describing algebraic structures on categories $\ch$ of presheaves with operations taking as input an ``arity diagram,'' such as strings of composable edges in a graph, and outputting a single cell, like the composite arrow in a category. These kinds of algebras include most familiar higher category structures, whose operations typically encode the structure of unit cells, composition of cells, or symmetries where the various ``sources'' and/or ``targets'' of a cell can be permuted to form a new cell.

The unit and multiplication transformations are expected to be cartesian to encode (as in the proof of \cref{rep_local_equivalence}) that the equations in these algebraic structures are always between operations with the same arity. For instance, associativity for categories asserts that any binary parenthesization of the same string of arrows has the same total composite; this is an equation between two potentially different operations with the same arity diagram.

\begin{ex}
\cite[Example C.3.3]{leinster} describes the familial structure of the free category monad on graphs, similar to \cref{cat_example} above. \cite[Proposition F.2.3]{leinster} shows that the free $n$-category monad is familial on $n$-globular sets, which are presheaves over $\G_n = 0 \starrows \cdots \starrows n$ with $\sigma \circ \sigma = \tau \circ \sigma$ and $\sigma \circ \tau = \tau \circ \tau$ at each level. The operations and arities in this case correspond to the $n$-dimensional free globular pasting diagrams, which generalize the strings $\atol{n}$ of composable arrows to diagrams of composable $n$-cells.
\end{ex}

\begin{ex}\label{multicats}
\cite[Example 2.14]{WeberPra} constructs the free symmetric multicategory monad on multigraphs. Multigraphs are presheaves over the category $\cat{M}$ with objects $0, (0,1), (1,1), (2,1),...$ and morphisms $\sigma_1,...,\sigma_n,\tau  \colon 0 \to (n,1)$ for each $n \ge 0$ with no nontrivial compositions. Cell diagrams $X$ over $\cat{M}$ look like graphs with vertices $X_0$ and $n$-to-1 edges $X_{(n,1)}$ with $n$ sources determined by $X_{\sigma_i}$ and a single target determined by $X_\tau$.

The familial functor of this monad has operations and arities corresponding to any ``tree-with-permutation,'' an arbitrarily large finite composition of these many-to-one edges given by gluing together a source vertex of one such ``multiedge'' to the target vertex of another, perhaps many times, and at the end permuting the remaining ``leaves'' of the tree, source vertices not connected to the target of another multiedge. These trees describe the possible composites of ``multiarrows'' in symmetric multicategories, which are precisely the algebras for the monad.

Similarly, \cite[Proposition 2.9]{Shapely} exhibits the free polycategory monad on polygraphs as familial, where polygraphs consist of vertices and $n$-to-$m$ edges for $n,m \in \nats$, and in a polycategory two such ``polyarrows'' can be composed along a single target vertex of the first shared with a single source vertex of the second, with suitable many-to-many analogues of unit and associativity axioms.
\end{ex}

\begin{ex}
A double graph is a diagram of vertices, two distinct types of edges drawn as $\rightarrowtail$ (``horizontal edges'') and $\circ\!\!\to$ (``vertical edges''), and squares with one pair of parallel edges of each type as below left:
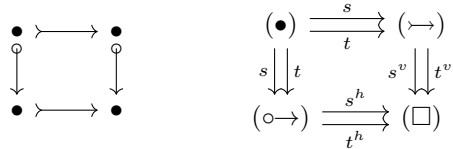
\bctikz 
\bullet \rar[tail] \dar["\circ" marking, pos=0] & \bullet \dar["\circ" marking, pos=0] \\ \bullet \rar[tail] & \bullet 
\ntikz
(\bullet) \rar[shift left=1]{s} \rar[shift right=1, swap]{t} \dar[shift left=1]{t} \dar[shift right=1, swap]{s} & (\rightarrowtail) \dar[shift left=1]{t^v} \dar[shift right=1, swap]{s^v}\\ (\circ\!\!\to) \rar[shift left=1]{s^h} \rar[shift right=1, swap]{t^h} & (\square)
\ectikz
Double graphs are precisely presheaves over the category above right with four objects corresponding to vertices, both types of edges, and squares, with relations $s^hs = s^vs$, $s^ht = t^vs$, $t^hs = s^vt$, $t^ht = t^vt$, or equivalently the cartesian product $\G_1 \times \G_1$.

Double categories are algebras for a familial monad on double graphs with a composition operation with arity each string of $n$ composable horizontal or vertical edges outputting an arrow of the same type, and on squares an operation with arity each $n \times m$ grid of squares ($n,m \in \nats$). Such an algebra has both its horizontal and vertical arrows form categories, and additionally has horizontal and vertical compositions of squares which satisfy the usual unit, associativity, and interchange equations as any composite of those operations must agree with the unique operation with the relevant grid as its arity. We describe how to formally specify the data of such a monad in \cref{monad_representation}, and discuss in \cref{cubical_omega_categories} how to do so for a similar monad to this one, for cubical presheaves with only one type of edge and cubes in arbitrarily high dimension.
\end{ex}

\section{Representability}\label{representability}

A familial representation reduces the data of a familial functor to the pair of $S$ which records the operations outputting cells of each type and $E$ which records the arities of those operations. In a familial monad, the unit specifies an ``identity operation'' ensuring that the functor preserves all of the existing cells of a presheaf, while the multiplication provides a means of composing the operations with each other in an appropriate sense. The goal of this section, and the paper as a whole, is to make this additional data precise in terms of the familial representation $(S,E)$.

First, we state and outline a proof of the general result that the 2-category $\Fam$ of familial functors is biequivalent to a bicategory $\Rep$ of familial representations. A direct proof of this statement would be straightforward but incredibly tedious, so we provide an alternative argument using categorical polynomials and Grothendieck constructions. The details of this proof are left to the appendices.

As a consequence, the equivalence $\Rep(\C,\C) \simeq \Fam(\ch,\ch)$ of the previous section is in fact an equivalence of monoidal categories, and the . Therefore to specify a monoid in $\Fam(\ch,\ch)$, which is precisely a familial monad, it suffices to produce a monoid in $\Rep(\C,\C)$, which induces a monad structure on the associated familial endofunctor. We describe precisely the data of such a \emph{monad representation}, and in the next section use this characterization of familial monads to describe new examples.

\subsection{The bicategory $\Rep$}

The equivalence of categories $H \colon \Rep(\C',\C) \to \Fam(\chp,\ch)$ in \cref{rep_local_equivalence}, ranging over any small categories $\C',\C$ and landing in the morphism categories of the 2-category $\Fam$, resembles the functors on morphism categories making up a bifunctor from a bicategory $\Rep$ with objects small categories and morphism categories given by $\Rep(\C',\C)$. Our main technical result is that this is in fact the case.

\begin{thm}\label{is_bicategory_equivalent}
There is a bicategory $\Rep$ with objects small categories, 1-cells familial representations, and 2-cells given by morphisms of representations, such that $H \colon \Rep \to \Fam$ sending a small category $\C$ to its presheaf category $\ch$ and acting on morphisms as described above is a biequivalence.
\end{thm}

\begin{proof}
It suffices to show that $\Rep$ is a bicategory and $H$ is a bifunctor, as $H$ is bijective on objects and an equivalence on morphism categories by \cref{rep_local_equivalence}.
To do so, we consider the diagram
$$\Rep \xrightarrow{\g} \Poly^{vf} \xrightarrow{P_d} \CAT,$$
where $P_d$ is a bifunctor (\cref{poly_bifunctor}) landing in familial functors between presheaf categories (but not necessarily cartesian transformations), $\g$ has the elements of a \emph{colax} bifunctor (\cref{is_colax}), and $\Poly^{vf}$ is the bicategory of \emph{very fibrous} categorical polynomials in $\Cat$, extending a sub-bicategory of Weber's categorical polynomials in $\Cat$ (\cite{WeberPoly}) to include the \emph{vertical} 2-cells introduced in \cite{GambinoKock}. 

We show (\cref{like_bifunctor}) that $P_d$ sends the colax structure maps of $\g$ to natural isomorphisms in $\Fam$, so that the composite $P_d \circ \g$ has the elements of a bifunctor. Furthermore, the composite (unlike $P_d$) lands in cartesian natural transformations (\cref{polynomials_cartesian}), so it factors through $\Fam$. The assignment $\Rep \to \Fam$ then sends  $\C$ to $\ch$ and agrees with $H$ on morphism categories (\cref{composes_to_H}), so $H$ has the elements of a bifunctor.

By ``elements of a (colax) bifunctor'', we mean that after describing the identities, composites, unitors, and associators of $\Rep$, but without directly proving that the triangle and pentagon laws hold, $\g$ and its colax structure maps are shown to satisfy the unitality and associativity equations of a colax bifunctor. The composite $H$ then has all of the elements of a bifunctor except for a complete proof that its domain $\Rep$ is a bicategory. To complete the proof, we recall (\cref{reflects_triangle}) that given such data with $H$ consisting of faithful functors on morphism categories, the triangle and pentagon equations for $\Rep$ can be deduced from those for $\Fam$ and the unitality and associativity equations for $H$.

\end{proof}

The appendices are devoting to defining $\Poly^{vf},P_d,\g$ and proving their properties asserted in this proof.

\subsection{Familial monad representations}

We have now established that $H \colon \Rep(\C,\C) \to \Fam(\ch,\ch)$ is a monoidal equivalence, which in particular induces an equivalence between the corresponding categories of monoids: 

\begin{thm}\label{monad_representation}
For any small category $\C$, $H$ restricts to an equivalence between familial monads on $\ch$ and monoids in the monoidal category $\Rep(\C,\C)$.
\end{thm}

We can therefore define a familial monad by giving a monoid in $\Rep(\C,\C)$, which we call a \emph{familial monad representation} or (depending on the context) a \emph{higher category schema}. Using the definition of $\Rep$ in \cref{rep_is_bicategory}, a familial monad representation consists of the following:
\begin{itemize}
\item A functor $S \colon \C^{op} \to \Set$
\item A functor $E \colon \sint S \to \ch$
\end{itemize}
These alone define the functor $T = \Hse \colon \ch \to \ch$ by 
$$TX_c = \coprod_{t \in Sc} \Hom(Et,X).$$
\begin{itemize}
\item A map $e^S$ from $S^0 \colon c \mapsto \{*_c\}$ to $S$ with an isomorphism $e^E \colon E \sint e^S \to E^0$, which amounts to natural isomorphisms $Ee(*_c) \cong y(c)$
\end{itemize}
This provides the unit map 
$$\eta_X \colon X_c \cong \Hom \blp y(c),X \brp \cong \Hom \blp Ee^S(*_c),X \brp \emb \coprod_{t \in Sc} \Hom(Et,X) = TX_c.$$
\begin{itemize}
\item A map $m^S$ from $SS \colon c \mapsto \coprod_{t \in Sc} \Hom(Et,S)$ to $S$ along with an isomorphism ${m^E \colon E \sint m^S \to EE}$, which amounts to natural isomorphisms $Em(t,f) \cong \colim_{x \in Et} Ef(x)$
\end{itemize}
This provides the multiplication map
$$TTX_c \cong \coprod_{(t \in Sc, f \colon Et \to S)} \Hom \blp \colim_{x \in Et} Ef(x),X \brp \cong \coprod_{(t,f)} \Hom \blp Em^S(t,f),X \brp \to \coprod_{t' \in Sc} \Hom(Et',X) = TX_c.$$
\begin{itemize}
\item Commutativity of the following unitality and associativity diagrams: 
\bctikz
S^0S \arrow{dr}[swap]{\lambda^S} \arrow{r}{e^S \cdot \id} & SS \arrow{d}{m^S} \\
 & S
\ntikz
SS \arrow{d}[swap]{m^S} & SS^0 \arrow{dl}{\rho^S} \arrow{l}[swap]{\id \cdot e^S} \\
 S
\ntikz
(SS)S \arrow{rr}{\alpha^S} \arrow{d}[swap]{m \cdot \id} & & S(SS) \arrow{d}{\id \cdot m} \\
SS \arrow{dr}[swap]{m} & & SS \arrow{dl}{m} \\
& S
\ectikz
along with the corresponding diagrams of isomorphisms relating $e^E,m^E,\lambda^E,\rho^E,\alpha^E$, which for brevity we defer to \cref{rest_of_equations}.
Here $\lambda^S,\rho^S,\alpha^S$ are defined (in \cref{rep_is_bicategory}) by
$$\lambda^S \blp *_c,t \colon y(c) \to S \brp = t \in Sc, \qquad\qquad \rho^S \blp t, ! \colon Et \to S^0 \brp = t \in Sc,$$
$$\alpha^S \blp (t,f \colon Et \to S), F \colon \colim_{x \colon y(c') \to Et} Ef(x) \to S \brp = \blp t,G \colon Et \to SS \colon x \mapsto (f(x),F_x) \brp,$$
where $F_x$ is the composite $Ef(x) \to \colim_{x' \colon y(c') \to Et} Ef(x') \xrightarrow{F} S.$
\end{itemize}
These ensure that $(T,\eta,\mu)$ satisfy the unit and associativity laws.

\begin{rem}\label{rigid}
In practice, many examples have $E$ land in rigid diagrams, objects in $\ch$ with no nontrivial automorphisms, in which case all isomorphisms in the class of the representing diagrams $Et$ are then unique, so it suffices to define $e \colon S^0 \to S$, $m \colon SS \to S$ and check that $E \sint e \cong E^0$ and $E \sint m \cong EE$ as properties rather than structure, and check only the diagrams for the $S$-parts above.  This is the case for all of the examples we consider.
\end{rem}  

To illustrate the convenience of this characterization of familial monads, we define several interesting monads in this fashion that illustrate the common themes of familial monads freely adding cells to a diagram resembling units, composites, and/or symmetries built from the underlying cell shapes. They also provide interesting examples of the ``theory'' associated to a familial monad:

\begin{defn}[{\cite[Definition 4.4]{WeberPra}}]\label{theory}
Given a familial monad $T$ on $\ch$ with representation $(S,E)$, the corresponding \emph{theory} is the category $\Theta_T$ defined as the full subcategory of $alg(T)$ spanned by objects of the form $TEt$ for $t \in Sc$, $c$ an object of $\C$.
\end{defn} 

By Weber's Nerve Theorem (\cite[Theorem 4.10]{WeberPra}), there is always a fully faithful functor $alg(T) \to \hat{\Theta}_T$ sending $A$ to $N_TA \colon TEt \mapsto \Hom_{alg(A)}(TEt,A)$, and a diagram $X$ in $\hat{\Theta}_T$ is isomorphic to a nerve if and only if it is a \emph{$\Theta_T$-model}, meaning each $X_{TEt}$ is the limit of the diagram:
$$\sint Et \to \C \cong \sint S^0 \atol{\sint e^S} \sint S \atol{E} \Theta_T \atol{X} \Set.$$

\section{Examples of Familial Monad Representations}\label{examples}

\cref{monad_representation} reduces the task of defining a familial monad to specifying its operations, their arities, how these relate to one another, the identity operations, and how operations compose. We demonstrate the convenience of this characterization with several less familiar examples exhibiting the three most common types of algebraic structure: units, symmetries, and compositions.

\subsection{Adding Degeneracies}

In most cases of interest, the category $\C$ of cell shapes is made up only of morphisms from ``lower dimensional'' to ``higher dimensional'' cell shapes, resembling inclusions into each shape of its lower dimensional faces.

\begin{defn}
A category $\C$ is \emph{direct} if it admits an identity-reflecting functor to the linear order $\cat{Ord}$ of ordinals regarded as a category. Concretely, this amounts to a ``degree'' function from $\Ob(\C)$ to ordinals such that each morphism in $\C$ strictly raises degree.
\end{defn}

When $\C$ is direct and all of the degrees are finite, the degree functor can always be taken to send each object $c$ to the length $n$ of the longest string of nontrivial composable morphisms $c_0 \to c_1 \to \cdots \to c_{n-1} \to c_n = c$ in $\C$.

\begin{ex}
The \emph{semicube category}, which we denote $\square_\d$, is the free monoidal category generated by a single object $\s^1$ and two maps $\d_0,\d_1 \colon \s^0 \to \s^1$ where $\s^0$ is the monoidal unit. We further denote $\s^n := \s^1 \otimes \scdots{n} \otimes \s^1$,
and 
$$\d^n_{i,\eps} := \id_1 \otimes \scdots{i} \otimes \id_1 \otimes \d_\eps \otimes \id_1 \otimes \scdots{n-1-i} \otimes \id_1 \colon \s^n \to \s^{n+1}$$
for $i=1,...,n$, $\eps = 0,1$. These maps are generators of $\square_\d$. 

$\s^n$ can be thought of as an $n$-dimensional cube with $\d^n_{i,0},\d^n_{i,1}$ respectively its front and back $(n-1)$-dimensional faces in the $i$th direction, where the edges of the cube are directed from front to back. We will also write $\s^n$ for the semicubical set represented by $\s^n$.

$\square_\d$ is direct, with degree functor sending $\s^n$ to $n$ and all of the morphisms inclusions of faces from a lower dimensional cube to a higher dimensional one.
\end{ex}

Direct categories are most often discussed in the context of Reedy categories.

\begin{defn}
A Reedy category is a category $\R$ equipped with wide subcategories $\R_+,\R_-$ such that $\R_+$ is direct, the morphisms in $\R_-$ are strictly degree-lowering, and every morphism factors uniquely as a morphism in $\R_-$ followed by a morphism in $\R_+$.
\end{defn}

\begin{ex}\label{cube_degeneracies}
The \emph{cube category} $\square_{\d\sigma}$ can be defined as the free monoidal category generated by objects $\s^0,\s^1$ (with $\s^0$ the monoidal unit), morphisms $\d_0,\d_1$ as above, and $\sigma \colon \s^1 \to \s^0$ with $\sigma\d_0 = \sigma\d_1 = \id_{\s^0}$. $\square_{\d\sigma}$ is generated by the maps $\d^n_{i,\eps}$ as above along with ``degeneracy maps''
$$\sigma^n_i  := \id_1 \otimes \scdots{i-1} \sigma \scdots{n-i-2} \id_1 \colon \s^n \to \s^{n-1}$$
satisfying appropriate relations, which can be found in \cite[Equation 2.1]{cubicalcategories}.

$\square_{\d\sigma}$ is a Reedy category with $\R_+$ the subcategory $\square_\d$ and $\R_-$ the subcategory generated by the maps $\sigma^n_i$. As any map between $\s^0$ and/or $\s^1$ factors uniquely as a map in $\R_-$ followed by a map in $\R_+$ (either potentially an identity), the same is true for monoidal products of these maps, which can be factored the same way in each component. 
\end{ex}

Our characterization of familial monads allows us to show the following by a simple construction, showing in particular that there is a monad on semicubical sets which freely adds in the degenerate cubes of a cubical set (presheaf on $\square_{\d\sigma}$), whose algebras are cubical sets and whose theory category is $\square_{\d\sigma}$.

\begin{prop}\label{reedy_monad}
For any Reedy category $\R$, $\rh$ is equivalent to the category of algebras for a familial monad on $\rhp$, which has $\R$ as its theory category.
\end{prop}

\begin{proof}
Let $\R$ be a Reedy category. We define a familial representation $(S,E)$ over $\R_+$ as follows:
\begin{itemize}
	\item $Sc = \coprodl_{b \in \Ob(\R)} \Hom_{\R_-}(c,b)$ for $c \in \Ob(\R)$
	\item For $i \colon c' \to c$ in $\R_+$ and $t \colon c \to b$ in $Sc$, $ti$ factors uniquely as $i't' \colon c' \to b' \to b$ with $i'$ in $\R_+$ and $t'$ in $\R_-$. We then set $(Si)(t) = t'$
	\item For $t \colon c \to b$ in $Sc$, $Et = y(b)$
	\item For $t,i,t',i'$ as above, $Ei_t$ is given by $y(i') \colon y(b') \to y(b)$
\end{itemize}

The induced familial endofunctor on $\rhp$ sends $X$ to $TX$ where 
$$TX_c = \coprod_{b \in \Ob(\R)} \Hom_{\R_-}(c,b) \times X_b,$$
adding a ``degenerate'' $c$-cell for each $b$-cell and degree-lowering map $t \colon c \to b$. In the example above, this amounts to adding for each $n$-cube a degenerate $m$-cube for each projection from the $m$-cube to the $n$-cube.

As $\R_+$ has no nontrivial isomorphisms as a direct category, each presheaf $y(b)$ is rigid, so following \cref{rigid} we can specify a monoidal structure on $(S,E)$ as follows:
\begin{itemize}
	\item $e \colon S^0 \to S$ sends $*_c$ to $\id \colon c \to c$ for each $c \in \Ob(\C)$
	\item $SSc \cong \coprodl_{t \colon c \to b} Sb$, and $m \colon SS \to S$ sends $(t \colon c \to b, t' \colon b \to a) \in SSc$ to $(t't \colon c \to a) \in Sc$
	\item Clearly $Ee(*_c) \cong y(c)$ and $Em(t \colon c \to b, t' \colon b \to a) = y(a) \cong \colim_{i \colon y(c') \to y(b)} E(Si(t'))$
\end{itemize}

An algebra of $T$ is a presheaf $A$ in $\rhp$ along with a map $TA \to A$, which amounts to functions $A_t \colon A_b \to A_c$ for each $t \colon c \to b$ in $\R_-$ such that:
\begin{itemize}
	\item $A_{\id} \colon A_c \to A_c$ is the identity (unit law) 
	\item $A_t \circ A_{t'} = A_{t't}$ (multiplication law)
	\item for $i \colon c' \to c$ in $\R_+$ $A_iA_t = A_{t'}A_{i'}$ for $i',t'$ as described above (naturality of algebra structure map)
\end{itemize}
This is precisely the data of a presheaf over $\R$, and a map of algebras corresponds similarly to a morphism in $\rh$.

The theory category $\Theta_T$ is the full subcategory of $\rh$ consisting of the free algebras on the representable presheaves $y(d)$ of $\rhp$. $Ty(d)$ is simply the cells of $y(d)$ with a $c$-cell added for each pair of a non-identity arrow $c \to b$ in $\R_-$ and an arrow $i \colon b \to d$ in $\R_+$, which correspond to the morphisms $c \to d$ in $\R$ that don't come from $\R_+$. $Ty(d)$ is therefore the representable presheaf $y(d)$ in $\rh$, so $\Theta_T$ agrees with $\R$ as the full subcategory of representables in $\rh$.
\end{proof}

\begin{ex}
Perhaps the most famous Reedy category is the simplex category $\Delta$, whose direct subcategory contains only the face maps between simplices. Presheaves on this subcategory are called semisimplicial sets, and \cref{reedy_monad} shows that there is a familial monad on semisimplicial sets which adds in the degeneracies needed to form a simplicial set, with $\Delta$ as its theory.
\end{ex}

\begin{ex}
A much simpler example comes from the Reedy category $\G_{1,r}$, with $\G_1 = 0 \starrows 1$ as its direct subcategory and a single map $\epsilon \colon 1 \to 0$ as $\R_-$, with $\epsilon \circ \sigma = \epsilon \circ \tau = \id_0$. Here \cref{reedy_monad} produces the monad on graphs adding a new self-loop to every vertex, whose algebras are reflexive graphs.
\end{ex}

\begin{rem}\label{factorization_systems}
While presheaves on direct categories commonly form the data of higher categorical structures, this construction in fact applies much more broadly. The degree raising and lowering properties of the factorization system in a Reedy category was never used in the construction, which therefore shows that for any category $\C$ with wide subcategories $\C',\C''$ such that each morphism factors uniquely as a map in $\C''$ followed by a map in $\C'$, there is a familial monad on $\chp$ whose algebras are presheaves over $\C$ and with $\C$ as its theory category.

In the reverse direction, starting with a familial monad $T$ on $\ch$ there is an ``active-inert'' factorization system on $\Theta_T$ such that the ``inert'' maps are precisely those arising from maps between the arity presheaves in $\ch$. Writing $\Theta_0$ for this subcategory, $\Theta_T$ is then also the theory of a familial monad on $\Thoh$, as discussed in \cite[Lemma 4.5]{WeberPra}.
\end{rem}

\subsection{Adding Symmetries}

In types of higher categories with particularly symmetric cell shapes, cells are often equipped with reflected or permuted versions of themselves. 

\begin{ex}
The \emph{cubical nerve} of a category $\C$ is the cubical set with $n$-cubes the commutative $n$-dimensional cube diagrams in $\C$, faces given by restriction to the appropriate lower dimensional subcubes, and degeneracies given by inserting identities in the appropriate direction. For any square as below left in $\C$, there is also a square as below right with the directions swapped.
\bctikz \cdot \rar{f} \dar[swap]{g} & \cdot \dar{h} \\ \cdot \rar[swap]{k} & \cdot 
\ntikz \cdot \rar{g} \dar[swap]{f} & \cdot \dar{k} \\ \cdot \rar[swap]{h} & \cdot \ectikz

This type of operation in a cubical set is called a \emph{symmetry}: a single cell is sent to another cell of the same shape with its faces permuted in some way. $n$-dimensional cubes in a cubical nerve admit symmetries for any permutation of the $n$ directions. Symmetries can be added as automorphisms of the objects in any variety of cube category: the semicube category, the cube category, even the category $\STheta$ described below which incorporates composites of cubical cells. 
\end{ex}

Symmetries of various sorts are also found in other types of higher categories. For instance, in a symmetric multicategory any $n$-to-1 arrow is equipped with additional $n$-to-1 arrows for every permutation of the inputs. Much like the familial monads discussed above for adding degenerate cells, symmetries can be freely added to a presheaf on $\C$ by a familial monad. The corresponding theory category is then constructed by adding new automorphisms to the objects of $\C$.

\begin{defn}[{\cite[Definition 2.1, 2.3]{generalizedreedy}}]
For $\C$ a small category, a \emph{crossed $\C$-group} is a functor $G \colon \C^{op} \to \Set$ equipped with a group structure on each set $Gc$ denoted by $(\cdot,e_c)$ and a left $Gc$-action on each set $\Hom_\C(c',c)$ denoted by $(-)_\ast$, such that for each $g,h \in Gc$ and $i \colon c' \to c$, $i' \colon c'' \to c'$ in $\C$:
\begin{enumerate}[a)]
	\item $g_\ast(i \circ i') = g_\ast(i) \circ (Gi(g))_\ast(i')$
	\item $g_\ast(\id_c) = \id_c$
	\item $Gi(g \cdot h) = (Gh_\ast(i))(g) \cdot Gi(h)$
	\item $Gi(e_c) = e_{c'}$
\end{enumerate}

The \emph{total category} $\C G$ has the same objects as $\C$, with morphisms $c' \to c$ of the form $(i,g)$ where $i : c' \to c$ in $\C$ and $g \in Gc'$. Identities are of the form $(\id_c, e_c)$ and composition is given by $(i,g) \circ (i',h) = (i \circ g_\ast(i'), Gi'(g) \cdot h)$.
\end{defn}

Crossed $\C$-groups are precisely the structure needed to describe a system of symmetries that can be freely added to a presheaf on $\C$ by a familial monad, and the theory category for this monad is given by the total category.

\begin{prop}\label{symmetries_monad}
For $G$ a crossed $\C$-group, $\widehat{\C G}$ is equivalent to the category of algebras for a familial monad on $\ch$, which has $\C G$ as its theory category.
\end{prop}

\begin{proof}
$\C G$ has a unique factorization system given by the subcategory of morphisms of the form $(i \colon c' \to c, e_{c'})$, which is isomorphic to $\C$, and the subcategory of morphisms of the form $(id_c, g)$. In particular, any morphism $(i \colon c' \to c,g)$ factors as $(i,e_{c'}) \circ (\id_{c'},g)$. The result then follows from \cref{factorization_systems}. 

Concretely, the representation of the monad has $S = G$ and $E$ sends $g \in Gc$ to $y(c)$ with its faces permuted by $g_\ast$. The unit and multiplication arise from the unit and multiplication in the groups $Gc$.
\end{proof}

\begin{ex}
Symmetries in cubical sets are described by the crossed cubical group $G \colon \square \to \Set$ defined as follows:
\begin{itemize}
	\item $G\s^n  = \Sigma_n$, the permutation group on $n$ elements
	\item $G\d_{i,\eps} \colon \Sigma_n \to \Sigma_{n-1}$ sends a permutation $\gamma$ on $\{1,...,n\}$ to the permutation on $\{1,...,n-1\}$ given by removing $i$ from the domain of $\gamma$ and reordering (e.g. $G\d_{2,\eps}(231) = (21)$ by removing the 2 and reordering) 
	\item $G\sigma_i \colon \Sigma_n \to \Sigma_{n+1}$ sends $\gamma$ to the permutation on $\{1,...,n+1\} \cong \{1,...,i,i',i+1,...,n\}$ treating $i,i'$ as a single element (e.g. $G\sigma_2(321) = (322'1) = (4231)$ by moving 2 and 2' together and then relabeling)
	\item $\gamma_\ast(\d_{i,\eps}) = \d_{\gamma(i),\eps}$ and $\gamma_\ast(\sigma_i) = \sigma_{\gamma(i)}$. This is how each symmetry permutes the faces (and degeneracies) of a cube
	\item It is straightforward to check that these satisfy the axioms of a crossed cubical group
\end{itemize}

This monad takes a cubical set $X$ and adds in a new $n$-cube for each permutation of the dimensions of each cube in $X_n$. An algebra $A$ for this monad is equipped with a choice of these symmetries: for each $\gamma \in \Sigma_n$ and each $n$-cube $a \in A_n$, a choice of cube $\gamma(a) \in A_n$ whose faces and degeneracies are those of $a$ permuted by $\gamma$. This is precisely a symmetric cubical set (in the language of \cite{cubicalvarieties}, a presheaf on the symmetric cube category $\mathbb{C}_{(we,\cdot)}$). Symmetries can be similarly added to any other type of cubical sets that doesn't already have them.

There is also a crossed cubical group with $G\s^n = \{1,\tau\}^n$, where each $\tau$ in the $i$th position reverses the source and target faces of the $n$-cube in the $i$th direction (as in, swaps $\d_{i,0},\d_{i,1}$ and does the same for all maps that factor through them). Algebras for the associated familial monad are cubical sets with reversals (in \cite{cubicalvarieties}, presheaves on $\mathbb{C}_{(w,')}$). 
\end{ex}

\begin{ex}
Symmetries are also often considered for permuting the sources of many-to-one arrows, such as in a multicategory. For the category $\cat{M}$ of the vertex and $n$-to-1 arrows in \cref{multicats}, a crossed $\cat{M}$-group can be defined with $G0$ the trivial group, $G(n,1) = \Sigma_n$, and each $\gamma \in \Sigma_n$ permuting the $n$ different source maps $0 \to (n,1)$. Algebras for this monad are symmetric multigraphs: multigraphs equipped with, for each $n$-to-1 edge and each permutation $\gamma$, a choice of $n$-to-1 edge with its sources permuted by $\gamma$ (a restriction of the symmetric polygraphs discussed in \cite[Section 2.4]{Shapely}). 

Dendroidal sets \cite{dendroidal} are presheaves over the tree category $\Omega$, which is the theory category for the free symmetric multicategory monad on multigraphs \cite[Example 4.19]{WeberPra}. There is similarly a planar tree category $\Omega_{planar}$ which is the theory category for the free non-symmetric multicategory monad. The above crossed group structure for permutations of sources in an $n$-to-1 arrow extends to a crossed $\Omega_{planar}$-group sending each tree of many-to-one arrows to its group of planar rearrangements, whose total category is equivalent to $\Omega$ \cite[Example 2.8]{generalizedreedy}. Hence denrdoidal sets are algebras for a familial monad on planar dendroidal sets.
\end{ex}

\begin{ex}
Crossed $\C$-groups were originally defined for $\C=\Delta$ \cite{crossedsimplicialgroups}, and any crossed simplicial group provides a monad for adding symmetries to simplicial sets. For instance, the crossed simplicial group sending $[n]$ to $\ints/(n+1)$ has as its total category Connes' cycle category $\Lambda$, so cyclic sets are algebras for a familial monad on simplicial sets.
\end{ex}

\subsection{Cubical $\omega$-Categories}\label{cubical_omega_categories}

The examples above all have representable arity diagrams, but we can also define familial monads on semicubical sets whose algebras have compositional structure similar to $n$-tuple categories.

\begin{defn}[{\cite[Definition 2.1]{cubicalcategories}}]
A \emph{cubical $\omega$-category} is a cubical set equipped with $n$ composition operations for $n$-cubes in the $n$ different directions satisfying unit (with respect to degeneracies), associativity, and interchange equations.
\end{defn}

When restricted to cubes in dimensions up to $n$, cubical $n$-categories resemble $n$-tuple categories \cite[Definition 2.1]{nfold}, but without the $n$ distinct types of arrows and resulting $\binom{n}{m}$ distinct types of $m$-cubes. 

\begin{ex}
For a category $\C$, its cubical nerve is a cubical $\omega$-category with composition of two compatible cubes given by composition in $\C$ of the arrows in the $i$th direction:
$$f \mapsto 
\begin{tikzcd} \cdot \rar[equals] \dar[swap]{f} & \cdot \dar{f} \\ \cdot \rar[equals] & \cdot \ntikz
\cdot \rar{f} \dar[swap]{h} & \cdot \rar{g} \dar[swap]{i} & \cdot \dar{j} \\ \cdot \rar[swap]{k} & \cdot \rar[swap]{\ell} & \cdot \end{tikzcd}
\mapsto
\begin{tikzcd} \cdot \rar{gf} \dar[swap]{h} & \cdot \dar{j} \\ \cdot \rar[swap]{\ell k} & \cdot \ectikz
\end{ex}

\begin{rem}
Cubical $\omega$-categories with connections, symmetries, etc. can be defined similarly using equations such as those for connections in \cite[Equation 2.6]{cubicalcategories}. The example above is in fact a cubical $\omega$-category with both symmetries and connections. \cite{cubicalcategories} shows that the category of cubical $\omega$-categories with connections is equivalent to the category of strict globular $\omega$-categories.
\end{rem}

\begin{defn}
The \emph{Day convolution product} $\otimes \colon \sqh_\d \times \sqh_\d \to \sqh_\d$ on semicubical sets is defined by left Kan extension of the functor 
$$\square_\d \times \square_\d \xrightarrow{\otimes} \square_\d \xrightarrow{y} \sqh_\d$$
along the product of Yoneda embeddings $\square_\d \times \square_\d \to \sqh_\d \times \sqh_\d$ and has the 0-cube $\s^0$ as a unit.
\end{defn}

\begin{ex}
Consider the string $\atol{k}$ of $k$ composable 1-cubes (arrows) in $\sqh_\d$.
The Day product $\atol{k_1} \otimes \cdots \otimes \atol{k_n}$ is the standard $k_1 \times \cdots \times k_n$ grid of $n$-cubes, where each zero among the natural numbers $k_i$ reduces the top dimension of the cubes in the grid by one. The inclusions $s,t$ from $\s^0$ to $\atol{k}$ sending the 0-cube to the source or target of the string of 1-cubes lets us define the source and target maps
$$s_i,t_i \colon \blp \atol{k_1} \otimes \cdots \otimes \hat{\atol{k_i}} \otimes \cdots \otimes \atol{k_n} \brp \to \blp \atol{k_1} \otimes \cdots \otimes \atol{k_n} \brp$$
of an $n$-dimensional grid in each of the $n$ directions.
\end{ex}

\begin{prop}
Cubical $\omega$-categories are algebras for a familial monad on semicubical sets.
\end{prop}

Cubical $\omega$-categories are just as well algebras for a familial monad on cubical sets, but we prefer to use semicubical sets and treat degeneracies as algebraic structure rather than part of the underlying data. We show this in the simplest case of no connections or symmetries, though the same is true in those settings as well by a more complicated construction; a description of the monad for cubical $\omega$-categories with connections can be found in \cite[Section 2.7]{Kachour}. Presenting this monad as familial provides an alternative proof of \cite[Proposition 9]{Kachour}, which shows that it is cartesian, and a familial representation of the free cubical $\omega$-category with connections monad would similarly suffice to prove \cite[Theorem 1]{Kachour}.

\begin{proof}
The higher category schema $(\square_\d,S,E,e,m)$ is defined as follows:
\begin{itemize}
	\item $Sn = \nats^n$ and $S\d^n_{i,\epsilon}$ is the projection map $\nats^n \to \nats^{n-1}$ omitting the $i$th component
	\item $E() = \s^0$, $E(k_1,...,k_n) = \blp \atol{k_1} \otimes \cdots \otimes \atol{k_n} \brp$, and $E$ sends the generating morphisms\\ ${\d^n_{i,0},\d^n_{i,1} \colon (k_1,...,\hat k_i,...,k_n) \to (k_1,...,k_n)}$ to $s_i,t_i$ respectively
	\item The map $e \colon S^0 \to S$ sends $*_n$ to $(1,...,1) \in Sn$, where $\s^n \cong E(1,...,1)$ as $\atol{1}$ is precisely $\s^1$ and $\s^1 \otimes \scdots{n} \otimes \s^1 \cong \s^n$
	\item $SSn \cong \coprodl_{(k_1,...,k_n)} \Hom \blp E(k_1,...,k_n),S \brp$, where as $S\d^n_{i,0} = S\d^n_{i,1}$ a map from $E(k_1,...,k_n)$ to $S$ is determined by its values in $S1 = \nats$ on the 1-cubes $(0,...,0,j \to j+1,0,...,0)$ in the $k_1 \times \cdots \times k_n$ grid. Therefore $SSn \cong \coprodl_{(k_1,...,k_n)} \nats^{k_1} \times \cdots \times \nats^{k_n}$, and we define $m \colon SS \to S$ by 
$$\blp (k_1,...,k_n),(\ell_{1,1},...,\ell_{1,k_1}),...,(\ell_{n,1},...,\ell_{n,k_n}) \brp \mapsto \blp \sum_{i=1}^{k_1} \ell_{1,i},...,\sum_{i=1}^{k_n} \ell_{n,i} \brp.$$
For such $(k,\ell)$, $\colim \blp \sint E(k_1,...,k_n) \to \sint S \xrightarrow{E} \sqh_\d \brp$ is the grid given by plugging an ${\ell_{1,j_1} \times \cdots \times \ell_{n,j_n}}$ grid into the $(j_1,...,j_n)$th cube in the grid $E(k_1,...,k_n)$, which is isomorphic to $E \blp \sum\limits_{i=1}^{k_1} \ell_{1,i},...,\sum\limits_{i=1}^{k_n} \ell_{n,i} \brp$ as desired. Pictured below is the grid for the assignment
		$$\blp (3,2),(2,1,3),(1,2) \brp \mapsto (2+1+3,1+2) = (6,3):$$
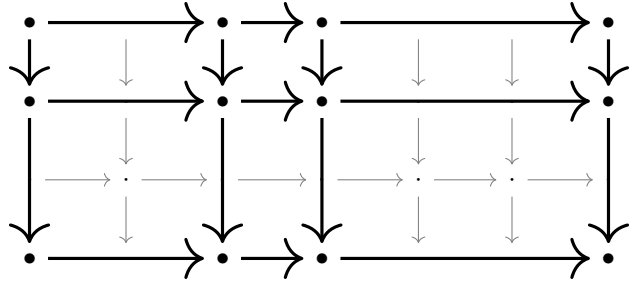
\bctikz
\bullet \ar[very thick]{rr} \ar[very thick]{d} & \cdot \dar[color=gray] & \bullet \ar[very thick]{r} \ar[very thick]{d} & \bullet \ar[very thick]{rrr} \ar[very thick]{d} & \cdot \dar[color=gray] & \cdot \dar[color=gray] & \bullet \ar[very thick]{d} \\
\bullet \ar[very thick]{rr} \ar[very thick]{dd} & \cdot \dar[color=gray] & \bullet \ar[very thick]{r} \ar[very thick]{dd} & \bullet \ar[very thick]{rrr} \ar[very thick]{dd} & \cdot \dar[color=gray] & \cdot \dar[color=gray] & \bullet \ar[very thick]{dd} \\
\cdot \rar[color=gray] & \cdot \rar[color=gray] \dar[color=gray] & \cdot \rar[color=gray] & \cdot \rar[color=gray] & \cdot \rar[color=gray] \dar[color=gray] & \cdot \rar[color=gray] \dar[color=gray] & \cdot \\
\bullet \ar[very thick]{rr} & \cdot & \bullet \ar[very thick]{r} & \bullet \ar[very thick]{rrr} & \cdot & \cdot & \bullet
\ectikz
	\item The unitality and associativity equations follow from the analogous properties of multi-valued sums, and rigidity of the arity diagrams
\end{itemize}

An algebra for $T$ is a semicubical set $A$ equipped with
\begin{itemize}
	\item Degeneracy maps as in \cref{cube_degeneracies}, where for each $1 \le i \le n$, $(1,...,1,0,1,...,1) \in Sn$ provides a map $s_i \colon A_{n-1} \to A_n$ satisfying the usual cubical identities
	\item $n$ binary composition operations for $n$-cubes, where for each $1 \le i \le n$, $(1,...,1,2,1,...,1) \in Sn$ provides a map $\mu_i \colon A_n \;{}_{d_{i,1}}\!\!\times_{d_{i,0}} A_n \to A_n$
	\item As for each composable grid of cubes up to dimension $n$, there is a unique element of $Sn$ sent to that grid by $E$, the multiplication law for the monad algebra $A$ ensures that these compositions are unital (with respect to degeneracies), associative in each direction, and satisfy the interchange law between compositions in different directions
\end{itemize}
which makes the structure of $A$ precisely that of a cubical $\omega$-category. 
\end{proof}

Uniqueness of each operation with respect to its arity makes $T$ \emph{shapely} in the sense of \cite{Shapely}.

\begin{rem}
The theory associated to $T$ has finite cubical grids as objects, with morphisms the homomorphisms of the free cubical $\omega$-categories $TE(k_1,...,k_n)$ generated by those grids. Concretely, these are the maps between the vertices of the grids which send rows in each direction of the domain to rows in the codomain, such as the map $TE(2,2) \to TE(3,5)$ depicted below:
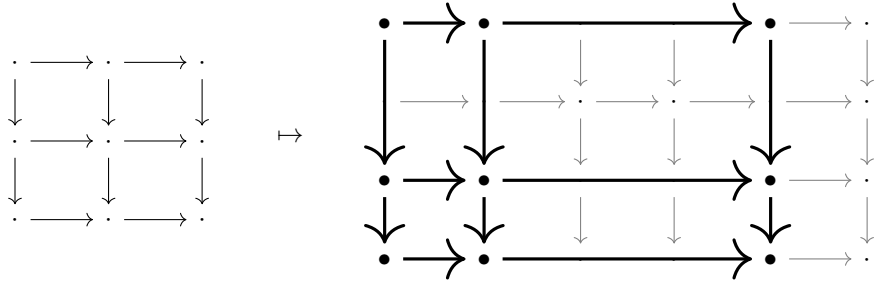
\bctikz 
\cdot \rar \dar & \cdot \rar \dar & \cdot \dar \\ \cdot \rar \dar & \cdot \rar \dar & \cdot \dar \\ \cdot \rar & \cdot \rar & \cdot \end{tikzcd}
\qquad \mapsto \qquad  \begin{tikzcd} 
\bullet \rar[very thick] \ar[very thick]{dd} & \bullet \ar[very thick]{rrr} \ar[very thick]{dd} & \cdot \dar[color=gray] & \cdot \dar[color=gray] & \bullet \rar[color=gray] \ar[very thick]{dd} & \cdot \dar[color=gray] \\ 
\cdot \rar[color=gray] & \cdot \rar[color=gray] & \cdot \rar[color=gray] \dar[color=gray] & \cdot \rar[color=gray] \dar[color=gray] & \cdot \rar[color=gray] & \cdot \dar[color=gray] \\ 
\bullet \rar[very thick] \dar[very thick] & \bullet \ar[very thick]{rrr} \dar[very thick] & \cdot \dar[color=gray] & \cdot \dar[color=gray] & \bullet \rar[color=gray] \dar[very thick] & \cdot \dar[color=gray] \\ 
\bullet \rar[very thick] & \bullet \ar[very thick]{rrr} & \cdot & \cdot & \bullet \rar[color=gray] & \cdot 
\ectikz

These grids are the cubical analogue of the pasting diagrams of Joyal's category $\Theta$, the theory of the free strict $\omega$-category monad on globular sets, and we call this category \emph{cubical $\Theta$}, written $\STheta$. Concatenation of the lists $(k_1,...,k_n)$ defines a tensor product $\otimes$ on $\STheta$, and all morphisms in $\STheta$ uniquely decompose under $\otimes$ into morphisms between the 1-dimensional grids $(k_1)$. The full subcategory of 1-dimensional grids is isomorphic to $\Delta$, so in this sense $\STheta$ is the free monoidal category generated by $\Delta$ with identity $(0)$ (in $\STheta$, $(0)$ is isomorphic to $()$).
\end{rem}

\begin{rem}
We can also consider the \emph{$n$-truncated} semicube category, defined as the full subcategory of $\square_\d$ spanned by $\s^0,...,\s^n$, where $\square_\omega$ recovers $\square_\d$. The construction above restricted to dimensions up to $n$ gives a monad on $n$-truncated semicubical sets whose algebras are the analogous notion of cubical $n$-categories, with theory category $\STheta_n$ of grids up to dimension $n$.
\end{rem}

\appendix

\section{Grothendieck Constructions}\label{grothendiecks}

Before discussing polynomials in $\Cat$ built out of Grothendieck fibrations and opfibrations, we review the relevant definitions and provide convenient constructions of pullbacks, composites, and distributivity pullbacks of opfibrations, then establish a basic functoriality result for two sided fibrations.

\subsection{Operations on Opfibrations}

Grothendieck fibrations and opfibrations are usually defined as functors between categories satisfying certain lifting properties allowing them to be equivalently defined in terms of their fibers over each object and morphism in the codomain. For convenience, we use this ``Grothendieck correspondence'' between such functors and their fibers, usually considered a theorem, as our definition of (op)fibrations:

\begin{defn}
A functor $p \colon \A \to \B$ is an \emph{opfibration} if it is (up to isomorphism of the domain) of the form $p_\Phi \colon \soint \Phi \to \B$ for some $\Phi \colon \B \to \Cat$, where $\soint \Phi$ is the following category: 
\begin{itemize}
	\item Objects are pairs $(b,x)$ for $b \in \Ob \B $, $x \in \Ob \Phi b$
	\item Morphisms are pairs $(i_0,i_1) \colon (b,x) \to (b',x')$ for $i_0 \colon b \to b'$ in $\B$, $i_1 \colon \Phi(i_0)(x) \to x'$ in $\Phi b'$
	\item The identity at $(b,x)$ is given by $(\id_b,\id_x)$
	\item Composites are given by $(i'_0,i'_1) \circ (i_0,i_1) = \blp i'_0 \circ i_0, i'_1 \circ \Phi(i'_0)(i_1) \brp$
	\item The functor $p_\Phi \colon \soint \Phi \to \B$ sends $(b,x)$ to $b$ and $(i_0,i_1)$ to $i_0$
\end{itemize}
For an opfibration $p$, written $p \colon \A \fib \B$, we write $\Phi_p$ for the corresponding functor $\B \to \Cat$, where $\Phi_p(b)$ is (up to isomorphism) the fiber of $p$ over the object $b$ in $\B$. 

Dually, \emph{fibrations} $\A \to \B$ are those functors corresponding to an analogous construction for functors $\B^{op} \to \Cat$, and \emph{discrete fibrations} correspond to functors $\B^{op} \to \Set \emb \Cat$, precisely the categories of elements for functors in $\bh$, hence the similar notation $\sint$.
\end{defn}

(Op)fibrations are often taken to be a more general class of functors corresponding instead to pseudofunctors $\B \to \Cat$, with those corresponding to strict functors called ``split opfibrations.'' For our purposes we can restrict to split opfibrations, though it is straightforward to extend the constructions on split opfibrations below to the more general setting. For a more thorough account of fibrations and opfibrations see \cite{grayfibred} or \cite[Appendix]{tholencolimit}.

\begin{rem}
$\soint \Phi$ can be seen as the ``lax colimit'' of $\Phi \colon \B \to \Cat$, in the sense that it is initial among categories with a lax cocone from $\Phi$ (see \cite[Remark 2.13]{tholencolimit}). In this case the lax cocone is given by the functors $J_b \colon \Phi(p) \emb \soint \Phi \colon x \mapsto (b,x)$ for each object $b$ in $\B$ and natural transformations $J_i \colon J_b \Rightarrow J_{b'} \circ \Phi(i)$ for each morphism $i \colon b \to b'$ in $\B$ with $x$ component given by 
$$(i,\id) \colon (b,x) \to \blp b',\Phi(i)(x) \brp.$$
As lax cocones generalize strict cocones, this universal property provides a canonical functor $Q_\Phi \colon \soint \Phi \to \colim(\Phi)$ from the lax colimit to the strict colimit of $\Phi$ in $\Cat$ (see \cite[Example 4.8]{tholencolimit}).
\end{rem}

While the following propositions are usually proven using equivalent definitions of opfibrations in terms of lifting properties (\cite[Proposition 3.1]{grayfibred}), our main result relies on a more explicit description of pullbacks and composites of opfibrations in terms of $\Cat$-valued functors. These constructions are widely understood but do not appear in the literature.

\begin{prop}\label{opfibration_pullback}
Opfibrations are closed under pullback.
\end{prop}

\begin{proof}
Consider the diagram below left in $\Cat$:
\bctikz
& \C \arrow{d}{f} \\
\A \arrow[two heads]{r}{p} & \B
\ntikz
D \rar{\psi} \dar[swap]{\psi} & \C \arrow{d}{u} \\
\A \arrow[two heads]{r}{p} & \B
\ectikz

We show that the opfibration $p_{(\Phi_p \circ u)} \colon \soint (\Phi_p \circ u) \fib \C$ is the pullback of $p$ along $u$. For any pair of functors $\phi,\psi$ commuting as above right, we define $\gamma \colon D \to \soint (\Phi_p \circ u)$ by
$$d \mapsto \blp \psi(d), \phi(d) \brp, \qquad (i \colon d \to d') \mapsto \blp \psi(i), \phi(i) \brp$$
where $\psi(d) \in \C$ and $\phi(d) = (b,x)$ for 
$$b=p\phi(d)=u\psi(d) \in \Ob \B, \qquad x \in \Ob\Phi_p(u\psi(d)).$$ 
$\psi(i) \colon \psi(d) \to \psi(d')$ in $\C$, and $\phi(i) = (i_0,i_1) \colon (b,x) \to (b',x')$ for
$$i_0 = p\phi(i) = u\psi(i) \textrm{ in } \B, \qquad i_1 \colon \Phi_p(i_0)(x) \to x' \textrm{ in } \Phi_p(u\psi(d')).$$ 
Clearly $p_{(\Phi_p \circ u)} \circ \gamma = \psi$, and likewise $\phi$ factors as $\gamma$ followed by the functor 
$$\soint u \colon \soint (\Phi_p \circ u) \to \soint \Phi_p \cong \A, \qquad \blp c \in \C, x \in \Phi_p(u(c)) \brp \mapsto \blp u(c), x \brp, \qquad (j_0,j_1) \mapsto \blp u(j_0),j_1 \brp.$$ 
\end{proof}

Recall the \emph{lax overcategory} $\Cat//\C$ with objects small categories over $\C$ and morphisms diagrams as below in $\Cat$:
\bctikz X \ar{rr}{} \ar{dr}[swap]{} \ar[Rightarrow, shorten=18, shift right=4]{rr}{} & & Y \ar{dl}{} \\ & \C \ectikz
We will also sometimes consider $\Cat//\C$ for $\C$ a large category, the \emph{lax slice} over $\C$ of the inclusion $\Cat \to \CAT$ of small into large categories, though for our purposes it will not make much difference whether $\C$ is large or small.

\begin{rem}\label{int_functoriality}
The assignment $u \mapsto \soint u$ in the proof above extends the assignment $\soint$ into a functor $\Cat/\Cat \to \Cat$. However, given functors 
$$u \colon \C \to \B, \quad \Phi \colon \C \to \Cat, \quad \Phi' \colon \B \to \Cat,$$
a functor $\soint \Phi \to \soint \Phi'$ commuting with $u$ corresponds to functors ${\phi_c \colon \Phi(c) \to \Phi'(u(c))}$ which are \emph{lax natural} in $c$, meaning they are equipped with natural transformations ${\phi_i \colon \Phi'(u(i))\phi_c \Rightarrow \phi_{c'}\Phi(i)}$ for each $i \colon c \to c'$, functorial in $i$. 
Given such data, we can define 
$$\soint(u,\phi) \colon \soint \Phi \to \soint \Phi', \qquad (c,x) \mapsto \blp u(c), \phi_c(x) \brp,$$
sending a map $\blp i_0 \colon c \to c',i_1 \colon \Phi(i_0)(x) \to x' \brp$ to 
$$\blp u(i_0), \Phi'(u(i_0))(\phi_c(x)) \xrightarrow{\phi_{i_0,x}} \phi_{c'}(\Phi(i_0)(x)) \xrightarrow{\phi_{c'}(i_1)} \phi_{c'}(x')\brp,$$
and it is straightforward to check that every such functor arises in this way. When $\phi$ is strictly natural, this shows that $\soint$ extends to a functor $\Cat//\Cat \to \Cat$
\end{rem}

\begin{prop}\label{opfibration_composite}
Opfibrations are closed under composition.
\end{prop}

\begin{proof}
Consider two opfibrations $p \colon \A \fib \B$, $q \colon \B \fib \C$. We define a functor $\Phi \colon \C \to \Cat$ and show that the corresponding opfibration agrees with $pq$. For an object $c$ of $\C$, let
$$\Phi(c) = \bigint \blp \Phi_q(c) \atol{J_c} \B \atol{\Phi_p} \Cat \brp.$$
For a morphism $i \colon c \to c'$, we have a morphism 
\bctikz \Phi_q(c) \ar{rr}{\Phi_q(i)} \ar{dr}[swap]{J_c} \ar[Rightarrow, shorten=18, shift right=6]{rr}{J_i} & & \Phi_q(c') \ar{dl}{J_{c'}} \\ & \B \dar[swap]{\Phi_p} \\ & \Cat \ectikz
in $\Cat//\Cat$. This assignment evidently respects identities and composites, so as $\soint$ is functorial over $\Cat//\Cat$, $\Phi$ defines a functor $\C \to \Cat$.

It remains to show that $\soint \Phi$ agrees with $pq$. We define an isomorphism $\A \cong \soint \Phi_p \cong \soint \Phi$ over $\C$ by, for $i_0 \colon c \to c'$ in $\C$, $i_1 \colon \Phi_q(i_0)(x) \to x'$ in $\Phi_q(c')$, and $i_2 \colon \Phi_p(i_0,i_1)(y) \to y'$ in $\Phi_p(x')$,
$$\blp (c,x),y \brp \mapsto \blp c,(x,y) \brp, \qquad \blp (i_0,i_1),i_2 \brp \mapsto \blp i_0,(i_1,i_2) \brp.$$
\end{proof}

Recall that for a fixed morphism $p \colon A \to B$ in a category $\A$ with pullbacks, there are functors $\sum_p \colon \A/A \adj \A/B \colon \delt_p$, where $\delt_p$ is defined by (choices of) pullback along $p$ and its left adjoint $\sum_p$ by postcomposition with $p$. $p \colon A \to B$ in $\A$ is \emph{exponentiable} if $\delt_p$ also has a right adjoint $\prod_p \colon \A/A \to \A/B$.

Weber shows (\cite[Section 2.2]{WeberPoly}) that $p$ is exponentiable if for all maps $u \colon X \to A$, there exists a terminal pullback square among those of the form:
\bctikz
Z \arrow{r} \arrow{d}[swap]{} & Y \arrow{dd}{\prod_p u} \\
X \arrow{d}[swap]{u} \\
A \arrow{r}{p} & B
\ectikz 
This square is called a \emph{distributivity pullback}, and given a choice of distributivity pullbacks $\prod_p u$ is defined as the map $Y \to B$ in $\A/B$.

\begin{prop}\label{opfibration_exponentiable}
Opfibrations are exponentiable in $\Cat$.
\end{prop}

This was proven in \cite[Corollary 6.2]{JohnstoneFibrations} using lifting properties, but as above we construct $\prod_p$ explicitly in terms of the functor $\Phi_p$ associated to an opfibration $p$.

\begin{proof}
Given $u \colon \X \to \A$, $p \colon \A \fib \B$ in $\Cat$, we construct the following distributivity square:
\bctikz
Z \arrow{d}[swap]{w} \arrow[two heads]{r}{q} & Y \arrow{dd}{v} \\
X \arrow{d}[swap]{u} \\
A \arrow[two heads]{r}{p} & B
\ectikz

An object in $Y$ is a pair $\blp b \in \Ob \B, f \colon \Phi_p(b) \to X \brp$ with $uf = J_b$, and a morphism $(b,f) \to (b',f')$ is a pair $\blp i \colon b \to b', \sigma \colon f \Rightarrow f'\Phi(i) \brp$ with $u\sigma = J_i$. The functor $v$ sends $(b,f)$ to $b$ and $(i,\sigma)$ to $i$. 

$q$ is the pullback of $p$ along $v$ given by \cref{opfibration_pullback}. Objects of $Z$ are then triples ${\blp b,f,x \in \Ob \Phi_p(b) \brp}$, and morphisms of $Z$ amount to triples $(i,\sigma,j)$ where $j \colon \Phi_p(i)(x) \to x'$ in $\Phi_p(b')$. The functor $w \colon Z \to X$ sends $(b,f,x)$ to $f(x)$ and $(i,\sigma,j) \colon (b,f,x) \to (b',f',x')$ to $f'(j) \circ \sigma_x$. $uw$ agrees with the projection map from \cref{opfibration_pullback} as each $f$ is a partial section of $u$.

To see that this square is terminal, consider a pullback $q' \colon Z' \fib Y'$ as below:
\bctikz
Z' \arrow{d}[swap]{w'} \arrow[two heads]{r}{q'} & Y' \arrow{dd}{v'} \\
X \arrow{d}[swap]{u} \\
\A \arrow[two heads]{r}{p} & \B
\ectikz
There is a functor $k \colon Y' \to Y$ over $\B$ sending $c$ in $Y$ to 
$$\Phi_p(v'(c)) \cong \Phi_q(c) \atol{J_c} Z' \atol{w'} X,$$
which is a partial section of $u$ as $uw'$ agrees with the projection functor from \cref{opfibration_pullback}, and defined similarly on morphisms.  There is also a functor $\ell \colon Z' \to Z$ over $X$ sending ${\blp c,x \in \Ob \Phi_p(v'(c)) \brp}$ to $\blp v'(c),k(c),x \brp$, so that $q\ell = kq'$, and $(\ell,k)$ are unique with respect to these properties. 
\end{proof}

\begin{rem}\label{preserves_fibrations}
If $u$ is a fibration, then for $i \colon b \to b'$ there are functors 
$$Fun_{/\A} \blp \Phi_p(b'),X \brp \to Fun_{/\A} \blp \Phi_p(b),X \brp$$ 
exhibiting $\prod_p u$ as the Grothendieck construction for fibrations of the functor 
$$\B^{op} \to \Cat \colon b \to Fun_{/\A} \blp \Phi_p(b),X \brp,$$ 
and if $u$ is a discrete fibration then each category $Fun_{/\A} \blp \Phi_p(b),X \brp$ is discrete, so $\prod_p$ preserves (discrete) fibrations. $\delt_p$ also preserves (discrete) fibrations, as does $\sum_p$ when $p$ is a (discrete) fibration. 
\end{rem}

\subsection{Two Sided Fibrations}

In the polynomials of the following sections, we consider opfibrations $\A \fib \B$ whose domain is equipped with a functor to another category $\C$, which is made up of compatible fibrations from each fiber. For $\Phi \colon \B \to \Cat$, the data of a functor $\soint \Phi \to \C$ is precisely that of an extension of $\Phi$ to $\Cat//\C$. We will be interested in the case when this $\Phi \colon \B \to \Cat//\C$ factors through the category $Fib(\C)$ (or $DFib(\C)$) of categories with a (discrete) fibration to $\C$ and functors which correspond to strict natural transformations in $Fun(\C^{op},\Cat)$. In $DFib(\C)$ these are all functors which commute strictly over $\C$, and in $Fib(\C)$ these are the functors over $\C$ which preserve the cartesian morphisms.

\begin{defn}
A \emph{(discrete) two sided fibration} from $\C$ to $\B$ is a diagram of the form
$$\C \xleftarrow{p_1} \A \atol{p_2} \B,$$
where $p_2$ is an opfibration such that $\Phi_p \colon \B \to \Cat$ is equipped with a lift along the forgetful functor $Fib(\C) \to Cat$ (resp. $DFib(\C) \to \Cat$). Abusing notation slightly, we also write $\Phi_p \colon \B \to Fib(\C)$ for this lift. 
\end{defn}

\begin{rem}
This definition agrees with the existing notion of (split) two-sided fibration (see \cite[Definition 2.3.4]{notionsfibration}), as $Fib(\C) \simeq Fun(\C^{op},\Cat)$ and therefore so $Fun(\B,Fib(\C)) \simeq Fun(\C^{op} \times \B, \Cat)$ as in the standard definition. This shows that two-sided fibrations could be defined dually to the above as a fibration over $\C$ with fibers opfibered over $\B$, and in particular that $p_1$ in the above definition is a fibration.
\end{rem}

The following will be useful for composing polynomials built from two-sided fibrations.

\begin{prop}\label{fibration_of_fibration}
$\soint$ extends to a functor $\Cat//Fib(\C) \to Fib(\C)$.
\end{prop}

\begin{proof}
As $p_1$ above is a fibration, it remains only to show that given 
\bctikz \B \ar{rr}{u} \ar{dr}[swap]{\Phi} \ar[Rightarrow, shorten=26, shift right=6]{rr}{\phi} & & \B' \ar{dl}{\Phi'} \\ & Fib(\C) \ectikz
the map $\soint(u,\phi) \colon \soint \Phi \to \soint \Phi'$ of \cref{int_functoriality} commutes over $\C$ and preserves cartesian morphisms. This follows from the same property of the functors $\phi_b \colon \Phi(b) \to \Phi'(u(b))$, as the cocartesian maps on both sides are sent to identities in $\C$ and all of the cartesian morphisms in $\soint \Phi$ are contained in some $\Phi(b)$ (by the same property of its objects).
\end{proof}

We will primarily be interested in obtaining discrete two sided fibrations, which in addition to naturally arising functors $\B \to DFib(\C) \simeq \ch$ can also be obtained from general two sided fibrations:

\begin{ex}\label{free_discrete}
Recall the functor $|-|_\C \colon Fib(\C) \to DFib(\C)$, left adjoint to the inclusion $DFib(\C) \to Fib(\C)$ and sending a fibration $u$ over $\C$ to the discrete fibration whose fibers are the sets of connected components of the fibers of $u$. $|-|$ is equivalently given by $(\pi_0)_\ast \colon Fun(\C^{op},\Cat) \to Fun(\C^{op},\Set)$, left adjoint to postcomposition with the inclusion $\Set \to \Cat$.

Given $\Phi \colon \B \to Fib(\C)$ corresponding to the two sided fibration $p$, we obtain $|\Phi| \colon \B \to Fib(\C) \to DFib(\C)$, and denote the corresponding discrete two sided fibration by $|p|$. The natural unit map $u \to |u|$ in $Fib(\C)$ induces a map $\Phi \to |\Phi|$ and accordingly $p \to |p|$, where in the latter the map $\pi \colon \A \to |\A|$ sends each element of the (intersection) fiber over $(c,b)$ to its connected component.
\end{ex}

\section{Very Fibrous Polynomials in $\Cat$}\label{polynomials}

We now describe a bicategory of ``very fibrous polynomials'' in $\Cat$, along with other convenient properties of polynomials which help facilitate the comparison of familial representations and familial functors.

\subsection{Special Classes of Polynomials}

Recall from \cite{WeberPoly} that a polynomial $p$ in $\Cat$ from $\C'$ to $\C$ is a diagram as below such that $p_2$ is exponentiable:
\bctikz
& \A \arrow{dl}[swap]{p_1} \arrow{r}{p_2} & \B \arrow{dr}{p_3} \\
\C' & & & \C
\ectikz

\begin{defn}
A polynomial $p$ in $\Cat$ is:
\begin{itemize}
	\item \emph{fibrous} if $p_2$ is an opfibration
	\item \emph{very fibrous} if $p_2$ is an opfibration and $p_3$ is a discrete fibration
	\item \emph{quasi-familial} if $(p_1,p_2)$ is a two sided fibration and $p_3$ is a discrete fibration
	\item \emph{familial} if $(p_1,p_2)$ is a discrete two sided fibration and $p_3$ is a discrete fibration
\end{itemize}

These properties form a hierarchy: familial $\implies$ quasi-familial $\implies$ very fibrous $\implies$ fibrous. The opfibration $p_2$ corresponds to a functor $\Phi_p \colon \B \to \Cat$, which we also use to denote the functor:
\begin{itemize}
	\item $\Phi_p \colon \B \to \Cat//\C'$ if $p$ is (very) fibrous
	\item $\Phi_p \colon \B \to Fib(\C')$ is $p$ is quasi-familial
	\item $\Phi_p \colon \B \to DFib(\C') \simeq \chp$ if $p$ is familial
\end{itemize}
\end{defn}

Recall (\cite[Section 3.2]{WeberPoly}) that given any polynomial $p$ in $\Cat$, its associated polynomial functor $P(p)$ is the composite 
$$\Cat/\C' \atol{\delt_{p_1}} \Cat/\A \atol{\prod_{p_2}} \Cat/\B \atol{\sum_{p_3}} \Cat/\C.$$
If $p_2$ is an opfibration then by \cref{opfibration_exponentiable} we have, for $X$ a category over $\C'$, $\prod_{p_2}\delt_{p_1}X$ is the category with:
\begin{itemize}
	\item objects pairs $\blp b \in \Ob \B, f \colon \Phi_p(b) \to X \brp$ with $f$ commuting over $\C$
	\item morphisms pairs $\blp i \colon b \to b',\sigma \colon f \Rightarrow f'\Phi_p(i) \brp \colon (b,f) \to (b,f')$, with $\sigma$ lying over the canonical transformation from $\Phi_p(b)$ to $\Phi_p(b')$ in $\C'$ 
\end{itemize}

When $p$ is very fibrous, each of these components of $P(p)$ preserves discrete fibrations: discrete fibrations are closed under pullback, exponentiation along an opfibration by \cref{preserves_fibrations}, and composition with a discrete fibration. As discrete fibrations are equivalent to presheaves, such a $P(p)$ therefore restricts to a functor between presheaf categories:

\begin{defn}\label{polynomial_functors}
For $p$ a fibrous polynomial as above, we associate to it the functor $P_d(p)$, defined as the composite
$$\chp \atol{\delt_{p_1}} \ah \atol{\prod_{p_2}} \bh \atol{\sum_{p_3}} \ch.$$
\end{defn}

Consider a diagram $X$ in $\chp$, represented below as a discrete fibration over $\C'$: 
\bctikz
& \cdot \arrow{d} \arrow[two heads]{r} & \prod_{p_2}\delt_{p_1}X \arrow{dd} \arrow[equals]{dr} \\
& \delt_{p_1}X \arrow{dl} \arrow{d} & & \sum_{p_3}\prod_{p_2}\delt_{p_1}X \arrow{dd} \\
X \arrow{d} & \A \arrow{dl}[swap]{p_1} \arrow[two heads]{r}{p_2} & \B \arrow{dr}{p_3} \\
\C' & & & \C
\ectikz
Unwinding the definitions, $P_d(p)(X) = \sum_{p_3}\prod_{p_2}\delt_{p_1}X$ has $c$-cells 
$$\coprod_{b \in \B_c} Fun_{/\C'} \blp \Phi_p(b),X \brp.$$
If $p$ is familial, $\Phi_p$ lands in $DFib(\C') \simeq \chp$ and we have
$$P_d(p)(X)_c \cong \coprod_{b \in \B_c} \Hom_\chp \blp \Phi_p(b),X \brp,$$
hence $P_d(p)$ is a familial functor.

\begin{prop}\label{all_familial}
For any very fibrous polynomial $p$, $P_d(p)$ is a familial functor.
\end{prop}

\begin{proof}
Let $p$ be quasi-familial; we can form the familial polynomial $|p|$, called its \emph{familial replacement}, by replacing $\Phi_p \colon \B \to Fib(\C')$ with $|\Phi_p| \colon \B \to DFib/\C'$. 
By the adjunction discussed in \cref{free_discrete}, we have for $X$ discretely fibered over $\C'$ 
$$Fun_{/\C'} \blp \Phi_p(b),X \brp \cong \Hom_{Fib(\C')} \blp \Phi_p(b),X \brp \cong \Hom_\chp \blp |\Phi_p|(b),X \brp,$$
natural in $b$, which establishes an isomorphism 
$$\prod\nolimits_{p_2}\delt\nolimits_{p_1} \cong \prod\nolimits_{|p|_2} \delt\nolimits_{|p|_1}.$$
As $|p|_3 = p_3$, this shows that $P_d(p) \cong P_d(|p|)$, so $P_d(p)$ is a familial functor.

The general proof for a very fibrous polynomial proceeds similarly by noting that the left adjoint $\Cat/\C' \to Fib(\C')$ to the inclusion functor extends to $\Cat//\C'$, but the only example we will need at this level of generality is the identity polynomial discussed below, so we do not discuss this further. 
\end{proof}

$P_d$ is full in the sense that it can recover any familial functor by a familial polynomial.

\begin{defn}\label{familial_polynomial} 
For any familial functor $\chp \to \ch$ with representation $(S,E)$, we can form the familial polynomial $\g(S,E)$ given by
\bctikz
& \soint E \arrow{dl}[swap]{p_1} \arrow[two heads]{r}{p_2} & \sint S \arrow{dr}{p_3} \\
\C' & & & \C
\ectikz
where $\soint E$ is the Grothendieck construction of the functor $\sint S \atol{E} \chp \cong DFib/\C'$. $P_d(\g(S,E))$ then agrees with the familial functor associated to $(S,E)$ by \cref{polynomial_functors}. 
\end{defn}

\begin{ex}\label{familial_identity_polynomial}
The representation of the identity functor on $\ch$ is given by $(S^0,E^0)$, where $S^0$ is the terminal functor $\C^{op} \to \Set$ and $E^0 \colon \sint S^0 \cong \C \to \ch$ is the Yoneda embedding (\cref{identity_familial}).  The Grothendieck construction of $E^0$ is then the discrete two sided fibration with $\Hom_\C(c',c)$ as the (intersection) fiber over $(c',c)$. Morphisms in the fiber over $c'$ are given by commuting triangles under $c'$, and morphisms from in the fiber over $c$ are given by commuting triangles over $c$, with general morphisms given by composites of these which form commutative squares (below left). $\g(S^0,E^0)$ is then isomorphic to the polynomial below right:
\bctikz
c' \dar \rar \ar{dr} & \cdot \dar \\ 
\cdot \rar & c
\ntikz
& \C^\to \arrow{dl}[swap]{dom} \arrow{r}{cod} & \C \arrow[equals]{dr}{} \\ 
\C & & & \C
\ectikz
\end{ex}

\subsection{Cartesian Morphisms}\label{polynomials_cartesian}

To extend $\g$ to a functor from familial representations to polynomials, we recall the definition of morphisms between polynomials. 

In \cite{WeberPoly}, morphisms $f$ between polynomials $p$ and $q$ are commuting diagrams of the following form:
\bctikz
& \A \arrow{dl}[swap]{p_1} \arrow{r}{p_2} \arrow{dd}[swap]{u_0} \arrow[phantom, shift left=1]{ddr}[very near start]{\lrcorner} & \B \arrow{dr}{p_3} \arrow{dd}{u_1} \\
\C' & & & \C \\
& \A' \arrow{ul}{q_1} \arrow{r}[swap]{q_2} & \B' \arrow{ur}[swap]{q_3}
\ectikz
which for fibrous polynomials amounts to a (pseudo) natural isomorphism $\Phi_p \cong \Phi_q \circ u_1$ in $\Cat//\C$ with components in the subcategory $\Cat/\C$.

These \emph{cartesian} morphisms between polynomials induce natural transformations between polynomial functors as follows:
$$\sum\nolimits_{p_3}\prod\nolimits_{p_2}\delt\nolimits_{p_1} \cong \sum\nolimits_{q_3}\sum\nolimits_{u_1}\prod\nolimits_{p_2}\delt\nolimits_{u_0}\delt\nolimits_{q_1} \cong \sum\nolimits_{q_3}\sum\nolimits_{u_1}\delt\nolimits_{u_1}\prod\nolimits_{q_2}\delt\nolimits_{q_1} \atol{\epsilon} \sum\nolimits_{q_3}\prod\nolimits_{q_2}\delt\nolimits_{q_1}$$
The first isomorphism comes from pseudofunctoriality of $\sum,\delt$, the second is the Beck-Chevalley isomorphism for the pullback square, and the final map is the counit of the adjunction $\sum_u \dashv \delt_u$. $\epsilon$ is cartesian, hence so is the induced natural transformation (\cite[2.1]{GambinoKock}).  

\begin{lem}\label{cartesian_ff}
$P \colon \Poly(\C',\C)_\times \to Poly(\Cat/\C',\Cat/\C)_\times$, where the codomain is the category of polynomial functors and cartesian transformations, is fully faithful.
\end{lem}

\begin{proof}
While $\Cat$ is not locally cartesian closed, as every functor $\C \to 1$ is exponentiable, $\Cat/\C$ has the same tensoring and enrichment as described in \cite[1.3]{GambinoKock}, which suffices to replicate the proof of \cite[Proposition 2.9]{GambinoKock} in this setting. It then remains only to observe that any strict natural transformation between $\Cat$-enriched functors is strong, so any cartesian transformation between polynomial functors is uniquely represented by a cartesian morphism of polynomials.
\end{proof}

As $\ch \simeq DFib(\C)$ forms a full subcategory of $\Cat/\C$, $P_d$ also sends cartesian morphisms to cartesian natural transformations, which for $(u_0,u_1)$ as above unwinds to the map
$$P_d(p)(X)_c \cong \coprod_{b \in \B_c} Fun_{/\C'} \blp \Phi_p(b),X \brp \to \coprod_{b' \in \B'_c} Fun_{/\C'} \blp \Phi_q(b'),X \brp \cong P_d(q)(X)$$
sending $\blp b,f \colon \Phi_p(b) \to X \brp$ to $\blp u_1(b), \Phi_q(u_1(b)) \atol{(u_0)_b^{-1}} \Phi_p(b) \atol{f} X \brp$.

\begin{defn}
Given a morphism $(\phi^S,\phi^E) \colon (S,E) \to (S',E')$ of familial representations from $\C'$ to $\C$, we have the following cartesian morphism of familial polynomials:
\bctikz
& \soint E \arrow{dl}[swap]{p_1} \arrow{r}{p_2} \arrow{dd}[swap]{\soint \phi^E} \arrow[phantom, shift left=1]{ddr}[very near start]{\lrcorner} & \sint S \arrow{dr}{p_2} \arrow{dd}{\sint \phi^S} \\
\C' & & & \C \\
& \soint E' \arrow{ul}{q_1} \arrow{r}[swap]{q_2} & \sint S' \arrow{ur}[swap]{q_3}
\ectikz
where the square is a pullback as $\phi^E$ restricts to an isomorphism $E(t) \cong E'(\phi^S(t))$ on the fibers of $p_2,q_2$.
\end{defn}

Every cartesian morphism between parallel familial polynomials arises uniquely in this manner, as functors $u_0 \colon \sint S \to \sint S'$ over $\C$ are in bijective correspondence with maps $u_0 \colon S \to S'$ in $\ch$, and a functor $u_1 \colon \soint E \to \soint E'$ commuting with the rest of such a diagram restricts to natural isomorphisms $\sint Et \to \sint E'u_0(t)$ over $\C'$, corresponding to isomorphisms $Et \to E'u_0(t)$ in $\chp$. 

$\g$ therefore extends to a fully faithful functor from $\Rep(\C',\C)$. Denoting by $\Poly^{vf}(\C',\C)_\times$ the category of very fibrous polynomials from $\C'$ to $\C$ and cartesian morphisms between them, we have now established the following.

\begin{prop}\label{composes_to_H}
$H \colon \Rep(\C',\C) \to \Fam(\C',\C)$ factors as
$$\Rep(\C',\C) \atol{\g} \Poly^{vf}(\C',\C)_\times \atol{P_d} \Fam(\C',\C),$$
where $\g$ is fully faithful and $P_d$ is surjective on objects.  
\end{prop}

\begin{rem}\label{reflects_isos}
As $\g$ is fully faithful and essentially surjective onto familial polynomials, the full subcategory $\Poly^{fam}(\C',\C)_\times$ of familial polynomials and cartesian morphisms is equivalent to $\Rep(\C',\C)$ and therefore $\Fam(\chp,\ch)$.
\end{rem}

\subsection{Vertical Morphisms}

While they are not included in the bicategory of polynomials in $\Cat$ defined in \cite{WeberPoly}, \cite{GambinoKock} describe a larger bicategory of polynomials and morphisms between them, in the more restrictive setting of a locally cartesian closed category. The additional morphisms are sent by an extension of $P$ to non-cartesian transformations of polynomial functors,
and morphisms from $p$ to $q$ admit a factorization system 
with right class the cartesian morphisms and left class the \emph{vertical morphisms}, given by diagrams of the following form:
\bctikz
 & \A \arrow{dl}[swap]{p_1} \arrow{r}{p_2} & \B \arrow[equals]{dd} \arrow{dr}{p_3} \\
\C' & & & \C \\
 & \A' \arrow{ul}{q_1} \arrow{uu}{v} \arrow{r}[swap]{q_2} & \B \arrow{ur}[swap]{p_3}
\ectikz
which for fibrous $p,q$ amounts to (by \cref{int_functoriality}) a lax natural transformation $\Phi_q \to \Phi_p$ in $\Cat//\C'$ with components in $\Cat/\C$ and lax structure lying over identities in $\C$.

If $v$ is exponentiable there is a transformation $P(p) \to P(q)$ given by
$$\sum\nolimits_{p_3}\prod\nolimits_{p_2}\delt\nolimits_{p_1} \atol{\eta} \sum\nolimits_{p_3}\prod\nolimits_{p_2}\prod\nolimits_v \delt\nolimits_v \delt\nolimits_{p_1} \cong \sum\nolimits_{p_3} \prod\nolimits_{q_2} \delt\nolimits_{q_1},$$
where $\eta$ is the unit of the adjunction $\delt_v \dashv \prod_v$ and the second isomorphism comes from pseudofunctoriality. However, when $p_2,q_2$ are opfibrations, the desired transformation can be defined for any $v$.

\begin{lem}\label{vertical_ff}
(Analogue of \cite[Proposition 2.8]{GambinoKock}) For fibrous polynomials $p,q$ as above, natural transformations $P(p) \Rightarrow P(q) \colon \Cat/\C' \to \Cat/\C$ that restrict to the identity on $\id_{\C'}$ correspond bijectively with maps $v$ as above.
\end{lem}

\begin{proof}
As $\id_{\C'}$ is terminal in $\Cat/\C'$ and $P(p)(\id_{\C'}) = P(q)(\id_{\C'}) \cong \B$, any such transformation lifts uniquely to a natural transformation 
$$\prod\nolimits_{p_2}\delt\nolimits_{p_1} \to \prod\nolimits_{q_2}\delt\nolimits_{q_1} \colon \Cat/\C' \to \Cat/\B,$$
so it suffices to show that such transformations $\beta$ correspond bijectively with lax natural transformations $\phi \colon \Phi_q \to \Phi_p$ strict over $\C$.

$\beta \colon \prod_{p_2}\delt_{p_1} \to \prod_{q_2}\delt_{q_1}$, restricted to the fiber over $b \in \Ob \B$, is a map
$$Fun_{/\C} \blp \Phi_p(b),X \brp \to Fun_{/\C} \blp \Phi_q(b),X \brp$$
natural in $X$, which by Yoneda is uniquely determined by a functor $\phi_b \colon \Phi_q(b) \to \Phi_p(b)$ over $\C$ (this is essentially the argument in \cite[Proposition 2.8]{GambinoKock}). To extend this correspondence to morphisms, for each $i \colon b \to b'$ in $\B$ $\beta$ requires a mapping, natural in $X$, from transformations as pictured below (right side) to natural transformations filling in the outer diagram (all over $\C$):
\bctikz 
\Phi_q(b) \rar{\phi_b} \ar{dd}[swap]{\Phi_q(i)} & \Phi_p(b) \ar{dd}[swap]{\Phi_p(i)} \ar{dr}{u} \ar[Rightarrow, shift left=6, shorten=14]{dd} \\
& & X \\
\Phi_q(b') \rar[swap]{\phi_{b'}} & \Phi_p(b) \ar{ur}[swap]{u'}
\ectikz
Such a mapping could arise from precomposition with a natural transformation $\phi_i \colon \Phi_p(i)\phi_b \Rightarrow \phi_{b'}\Phi_q(i)$ (satisfying coherence conditions corresponding to functoriality in $i$). In a Yoneda style argument, we can set 
$$X = \Phi_p(b'), \quad u = \Phi_p(i), \quad u' = \id_{\Phi_p(b')}$$
in the diagram above and apply such a mapping to the identity transformation on $\Phi_p(i)$ to recover a transformation $\Phi_p(i)\phi_b \Rightarrow \phi_{b'}\Phi_q(i)$, and it is straightforward to check that these transformations satisfy the desired coherence conditions and provide a bijective correspondence between such $\beta$'s and $\phi$'s.
\end{proof}

\begin{rem}
The analogous extension of $P_d$ sends $v$ to the natural transformation
$$P_d(p)(X)_c \cong \coprod_{b \in \B_c} Fun_{/\C'} \blp \Phi_p(b),X \brp \to \coprod_{b' \in \B_c} Fun_{/\C'} \blp \Phi_q(b'),X \brp \cong P_d(q)(X)$$
mapping $\blp b,f \colon \Phi_p(b) \to X \brp$ to $\blp b,\Phi_q(b) \atol{v_b} \Phi_p(b) \atol{f} X \brp$.
\end{rem}

Just as in \cref{cartesian_ff}, the tensoring of $\Cat/\C$ over $\Cat$ lets us replicate the proof of \cite[Proposition 2.4]{GambinoKock} in this setting.  Therefore, as all strict natural transformations in a $\Cat$-enriched category are strong, 
the functor $Poly(\Cat/\C',\Cat/\C) \to \Cat/\C$ given by evaluating a polynomial functor or cartesian transformation at the terminal object $\id_{\C'}$ is a fibration. The cartesian maps with respect to this fibration are the cartesian transformations, and vertical maps are those whose component at $\id_{\C'}$ is the identity.

The factorization system on cartesian and vertical maps with respect to this fibration provides a canonical means of commuting past each other cartesian and vertical morphisms of polynomials from $\C'$ to $\C$, which uniquely represent cartesian and vertical transformations, respectively, of the corresponding functors (\cref{cartesian_ff}, \cref{vertical_ff}). Any composite of cartesian and vertical morphisms then has a unique factorization as a vertical morphism followed by a cartesian morphism, which suffices to define the category $\Poly^f(\C',\C)$ of fibrous polynomials from $\C'$ to $\C$ and general morphisms between them of this form, with full subcategory $\Poly^{vf}(\C',\C)$ of very fibrous polynomials.

\begin{rem}\label{vertical_iso}
The transformation induced by a vertical morphism is only cartesian when the map $v$ is an isomorphism, in which case it is identified with the cartesian morphism with $(u_0,u_1)$ given by $(v^{-1},id)$.
\end{rem}

We have now extended $P$ to a functor $\Poly^f(\C',\C) \to Poly(\Cat/\C',\Cat/\C)$, and $P_d$ to a functor $\Poly^{vf} \to Fam(\chp,\ch)$.

\begin{ex}\label{familial_identitor}
Consider the familial polynomial $\g(S^0,E^0)$ of \cref{familial_identity_polynomial}. We have the following vertical morphism $\epsilon$ from $\g(S^0,E^0)$ to the identity polynomial on $\C$
\bctikz
& \C^\to \arrow{dl}[swap]{dom} \arrow{r}{cod} & \C \arrow[equals]{dr} \arrow[equals]{dd}{} \\
\C & & & \C \\
& \C \arrow[equals]{ul}{} \arrow[equals]{r}[swap]{} \arrow{uu}{\id_{(-)}} & \C \arrow[equals]{ur}[swap]{}
\ectikz
where the vertical map (which is not exponentiable) sends an object $c$ in $\C$ to its identity morphism $\id_c$ in $\C^\to$. 
The transformation 
$$P_d(\epsilon) \colon P_d \blp \g(S^0,E^0) \brp (X)_c \cong Fun_{/\C} \blp \sint y(c),X \brp \to Fun_{/\C}(*,X) \cong P_d(1_\C)(X)_c$$
is induced by the map $* \to \sint y(c) \cong \C/c$ picking out $\id_c$, and by Yoneda this is an isomorphism. $\g(S^0,E^0)$ is the familial replacement of $1_\C$, in the sense of \cref{all_familial}.
\end{ex}

\begin{ex}\label{familial_productor}
For a quasi-familial polynomial $p$, the map $\pi \colon \A \to |\A|$ forms a vertical morphism of polynomials from $|p|$ to $p$, which $P_d$ sends to an isomorphism by \cref{all_familial}, $\pi$ being the unit of the adjunction between two sided fibrations and discrete two sided fibrations. This will be the key to comparing the compositions of familial representations and familial polynomials.
\end{ex}

\subsection{Bicategory Structure}

\cite{WeberPoly} constructs a bicategory whose objects are small categories with morphism categories $\Poly(\C',\C)_\times$, and a bifunctor $\Poly_\times \to \CAT$ sending $\C$ to $\Cat/\C$ and each polynomial to the corresponding polynomial functor, likewise for cartesian transformations. The bicategory structure of polynomials in \cite{GambinoKock} which also includes vertical morphisms, while only claimed for locally cartesian closed categories, agrees on cartesian morphisms with that of \cite{WeberPoly}. 

Extending Weber's bifunctor to a bicategory of polynomials on $\Cat$ with vertical morphisms would require, at a minimum, restricting vertical morphisms to those whose map $\A' \to \A$ is exponentiable. To avoid this restriction, we instead work with the sub-bicategory of (very) fibrous polynomials, between which all vertical morphisms induce unique transformations of the corresponding polynomial functors. 

\begin{lem}\label{fibrous_composite}
Fibrous, very fibrous, and quasi-familial polynomials are each closed under polynomial composition.
\end{lem}

\begin{proof}
Recall (\cite[1.11]{GambinoKock}, \cite[Definition 3.1.7]{WeberPoly}) that for composable polynomials $p,q$ their composite is defined as the outer diagram below, where all squares are pullbacks and the pentagon on the right is a distributivity pullback:
\bctikz 
 & & & \bar \A' \arrow{rr}{\bar q_2} \arrow{ddll}[swap]{\bar p_1} & & \bar \A \arrow{r}{\bar p_2} \arrow{dl}[swap]{} & \bar \B \arrow{dd}{\bar q_3} \\
 & & \;\; & & \bar \B' \arrow{dl}[swap]{} \arrow{dr}{} \\
 & \A' \arrow{rr}{q_2} \arrow{dl}[swap]{q_1} & & \B' \arrow{dr}{q_3} & & \A \arrow{r}{p_2} \arrow{dl}[swap]{p_1} & \B \arrow{dr}{p_3} \\
\C'' & & & & \C' & & & \C
\ectikz

This definition requires a choice of pullbacks and distributivity pullbacks in $\Cat$, which is provided by the constructions in \cref{opfibration_pullback} and \cref{opfibration_exponentiable}. 

For $p,q$ fibrous, $\bar p_2\bar q_2$ is an opfibration by \cref{opfibration_pullback} and \cref{opfibration_composite}. If they are very fibrous, $\bar q_3$ is a discrete fibration by \cref{preserves_fibrations}, hence so is $p_3\bar q_3$.

Now assume $p,q$ are quasi-familial. $\bar \B$ is precisely $P_d(p)(\B')$, so its objects are of the form $\blp b,f \colon \Phi_p(b) \to \B' \brp$, where $f$ commutes over $\C'$.  Using, \cref{opfibration_pullback} and \cref{opfibration_composite}, $\bar p_2\bar q_2$ corresponds to the functor $\Phi \colon \bar \B \to \Cat$ sending $(b,f)$ to 
$$\bigint \blp\Phi_p(b) \atol{f} \B' \atol{\Phi_q} Fib(\C'')\brp,$$ 
where $\Phi$ factors through $Fib(\C'')$ by \cref{fibration_of_fibration} as a morphism in $\bar \B$ from $(b,f)$ to $(b',f')$ is a natural transformation $f \Rightarrow f' \circ \Phi_p(i)$ over $\C'$ for $i \colon b \to b'$ (which in fact must be the identity by discreteness of $q_3$). Therefore $(q_1\bar p_1,\bar p_2 \bar q_2)$ forms a two sided fibration.
\end{proof}

\begin{rem}\label{explicit_composite}
As discussed above, in the composite fibrous polynomial $\bar \B$ consists of diagrams in $\B'$ indexed by a fiber of $p_2$, and the fiber in $\bar \A'$ over such a diagram is the lax colimit of the corresponding fibers of $q_2$.  In the next section we use the comparison map from each such lax colimit to a strict colimit to analyze composites of familial polynomials, which are not closed under composition.
\end{rem}

\begin{thm}\label{poly_is_bicategory}
Small categories, fibrous polynomials between them, and general morphisms of polynomials form a bicategory $\Poly^f$ under polynomial identities and composition, with a bifunctor $P \colon \Poly^f \to \CAT$ sending $\C$ to $\Cat/\C$ and acting on morphism categories by $P \colon \Poly^f(\C',\C) \to \CAT(\Cat/\C',\Cat/\C)$. Very fibrous polynomials form a sub-bicategory $\Poly^{vf}$.
\end{thm}

\begin{proof}
As fibrous polynomials are closed under polynomial identities and composition which agree with identities and composition of polynomial functors, it suffices to define horizontal composites of morphisms of polynomials and show that it agrees in the appropriate sense with horizontal composition in $\CAT$. These constructions can proceed exactly as in \cite{GambinoKock} by a transport argument, as by \cref{cartesian_ff}, \cref{vertical_ff}, and the analogue of \cite[Proposition 2.4]{GambinoKock} we have the analogue of \cite[Lemma 2.15]{GambinoKock}, upon which these constructions rely.

As the inclusion $\Poly^{vf}(\C',\C) \emb \Poly^f(\C',\C)$ is full, to show that $\Poly^{vf}$ is a subcategory is suffices to note that very fibrous polynomials include the identities and are closed under horizontal composition.
\end{proof}

\begin{cor}\label{poly_bifunctor}
There is a bifunctor $P_d \colon \Poly^{vf} \to \CAT$ sending a small category $\C$ to $\ch$ and all polynomials to familial functors.
\end{cor}

\begin{proof}
$P_d$ is constructed as a sub-bifunctor of $P$ given by restricting the categories $\Cat/\C$ to discrete fibrations over $\C$, the category of which is equivalent to $\ch$. That $P_d$ sends all polynomials to familial functors follows from \cref{all_familial}.
\end{proof}

Note that $P_d$ does not land in $\Fam$ as vertical morphisms are not necessarily sent to cartesian transformations, though as discussed below the vertical morphisms we are interested will be sent to isomorphisms.

\begin{ex}\label{identitor_composite}
Consider the composition of a fibrous polynomial $p$ with the familial polynomial $|1_\C| \cong \g(S^0,E^0)$:
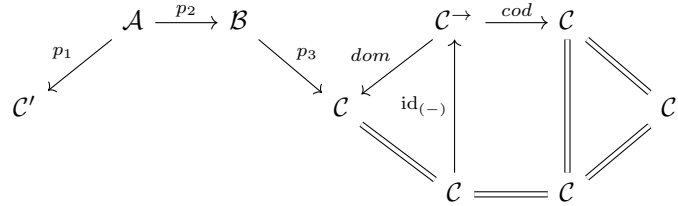
\bctikz
& \A \arrow{dl}[swap]{p_1} \arrow{r}{p_2} & \B \arrow{dr}{p_3} && \C^\to \arrow{dl}[swap]{dom} \arrow{r}{cod} & \C \arrow[equals]{dr} \arrow[equals]{dd}{} \\
\C' &&& \C & & & \C \\
&&&& \C \arrow[equals]{ul}{} \arrow[equals]{r}[swap]{} \arrow{uu}{\id_{(-)}} & \C \arrow[equals]{ur}[swap]{}
\ectikz
The resulting composite opfibration $\bar \A \to \bar \B$ with corresponding functor $\Phi \colon \bar \B \to \Cat//\C'$ has $\bar \B = P_d(|1_\C|)(\B) \cong \B$ by \cref{familial_identitor}, and for $b$ a $c$-cell of $\B$, 
$$\Phi(p) = \bigint \blp \C/c \cong \sint y(c) \atol{b} \B \atol{\Phi_p} \Cat//\C\brp.$$
Up to the unitor of $\Poly^f$, $\epsilon \cdot \id_p$ can be factored as the vertical followed by cartesian morphism:
\bctikz
& \bar \A \arrow{dl}[swap]{} \arrow{r}{} & \bar \B \arrow{dr}{} \arrow[equals]{d}{} \\
\C' & \A \arrow{l}{} \arrow{r}[swap]{} \arrow{u}[swap]{(\id_{(-)},-)} \arrow[equals]{d} & \bar \B \arrow{r}[swap]{} \arrow[equals]{d}{\wr} & \C \\
& \A \arrow{ul}{p_1} \arrow{r}[swap]{p_2} & \B \arrow{ur}[swap]{p_3}
\ectikz
where the natural transformation $\Phi_p \Rightarrow \Phi$ corresponding to the vertical map sends $x$ in $\Phi_p(b)$ to $(\id_c,x)$, with the evident action on morphisms.

The composite of $|1_{\C'}|$ and $p$ is then the opfibration $\bar \A \to \bar \B$ with corresponding functor $\Phi \colon \bar \B \to \Cat//\C'$, where $\bar \B = P_d(p)(\id_{\C'}) \cong \B$ as $\prod_{p_2}\delt_{p_1}$ as a right adjoint preserves terminal objects, and
$$\Phi(p) = \bigint \blp\Phi_p(b) \to \C' \atol{\C'/-} \Cat//\C'\brp.$$
Up to the unitor of $\Poly^f$ and a cartesian morphism like above, $\id_p \cdot \epsilon \colon |1_{\C'}|p \to p$ is the vertical morphism given by the transformation $\Phi_p \Rightarrow \Phi$ sending $x$ in $\Phi_p(b)$ lying over $c$ in $\C'$ to $(x, \id_c)$, with the evident action on morphisms.
\end{ex}

\section{Bicategory of Familial Representations}\label{rep_is_bicategory}

We now describe a bicategory structure $\Rep$ for familial representations, with small categories as objects and $\Rep(\C',\C)$ as morphism categories, while simultaneously assembling the functors $\g \colon \Rep(\C',\C) \to \Poly^f(\C',\C)$ into an identity-on-objects colax bifunctor $\g \colon \Rep \to \Poly^{vf}$.  The colax coherence maps for $\g$ are sent to isomorphisms by $P_d \colon \Poly^f \to \CAT$, endowing the composite $P_d \g \colon \Rep \to \CAT$, sending a representation to its associated familial functor, with the structure of a bifunctor.

\subsection{Identity and Identitor}

The identity representation in $\Rep(\C,\C)$ is given by $(S^0,E^0)$. 
The identitor of $\g$ at $\C$ is the (vertical) transformation in $\Poly^f$ from \cref{familial_identitor}:
\bctikz
 & \C^\to \arrow{dl}[swap]{dom} \arrow{r}{cod} & \C \arrow[equals]{dd} \arrow[equals]{dr} \\
\C & & & \C \\
 & \C \arrow[equals]{ul} \arrow{uu}[swap]{\id_{(-)}} \arrow[equals]{r} & \C \arrow[equals]{ur}
\ectikz

The identitor $\epsilon$ goes from $\g(S^0,E^0)$ to $1_\C$ and is not invertible in $\Poly^f(\C,\C)$, though by \cref{familial_identitor} it is sent to an isomorphism by $P_d$.  The direction of $\epsilon$ is opposite that of a \emph{lax} bifunctor, so as the same holds for the productor below, $\g$ will be a \emph{colax} bifunctor, which we prove in the following subsections.

\subsection{Product and Productor}

In \cref{familial_composite}, we showed that for familial functors $F \colon \chp \to \ch$ and $G \colon \chpp \to \chp$ with representations $(S,E)$ and $(S',E')$, their composite has representation $(SS',EE')$, where 
$$SS'_c = \coprod_{t \in Sc} \Hom(Et,S'), \qquad EE'(t,f) = \colim_{x \colon y(c') \to Et} E'f(x).$$  
Let this define the horizontal composition in $\Rep$, where both of these formulas are functorial in $S,E,S',E'$.

From the proof of \cref{fibrous_composite}, the polynomial composite $\g(S,E)\g(S',E')$ is of the form
\bctikz
& \A \arrow{dl}[swap]{p_1} \arrow[two heads]{r}{p_2} & \sint SS' \arrow{dr}{p_3} \\
\C'' & & & \C
\ectikz
where $\Phi_p(t,f \colon Et \to S')$ is given by
$$\bigint \blp\sint Et \atol{f} \sint S' \atol{E'} \chpp \cong DFib(\C'')\brp.$$
$\g(S,E)\g(S',E')$ is quasi-familial by \cref{fibrous_composite} but not familial, as $\Phi_p(t,f) \to \C''$ is not a discrete fibration: the fiber over $c''$ includes nontrivial morphisms of the form 
$$(i_x,\id) \colon \blp Ei_t(x),y \brp \to \blp x,E'f(i_x)(y) \brp$$ 
for $i \colon d \to c$ in $\C$, $x \in Et_c$, and $y \in E'f(Ei_t(x))_{c''}$. 

\begin{prop}\label{familial_replacement}
The familial replacement $|\g(S,E)\g(S',E')|$ is isomorphic to $\g(SS',EE')$.
\end{prop}

\begin{proof}
We have the following chain of isomorphisms 
$$P_d \blp \g(SS',EE') \brp \cong P_d \blp \g(S,E) \brp P_d \blp \g(S',E') \brp = P_d \blp \g(S,E)\g(S',E') \brp \cong P_d \blp |\g(S,E)\g(S',E')| \brp$$
by \cref{familial_composite}, bifunctoriality of $P_d$, and \cref{all_familial} respectively. As $P_d$ restricted to familial polynomials reflects isomorphisms (\cref{reflects_isos}), ${\g(SS',EE') \cong |\g(S,E)\g(S',E')|}$.
\end{proof}

More conceptually, the corresponding $\Phi_{|p|_2}(t,f)_c$ is the set of connected components of $\soint (E' \circ f)$ over $c$, which are precisely the $c$-cells of $EE'(t,f) = \colim_{x \colon y(c') \to Et} E'f(x)$, and ${(\pi_0)_\ast \colon \soint (E' \circ f) \to \colim (E' \circ f)}$ is the canonical map from the lax colimit of $E' \circ f$ to the strict colimit. The corresponding vertical map $\pi$, pictures below, is the productor $\g(SS',EE') \to \g(S,E)\g(S',E')$.
\bctikz
 & \soint EE' \arrow{dl}[swap]{} \arrow{r}{} & \sint SS' \arrow[equals]{dd} \arrow[]{dr} \\
\C'' & & & \C \\
 & \A \arrow{ul} \arrow{uu}[swap]{\pi} \arrow{r} & \sint SS' \arrow{ur}
\ectikz
By \cref{familial_productor}, $P_d$ sends $\pi$ to an isomorphism.

\subsection{Unitors and Unitality}

The left unitor $\lambda$ in $\Rep$ sends $\blp *_c, t \colon y(c) \to S \brp \in S^0Sc$ to $t \in Sc$ and has as its $E$-part the canonical $\colim_{j \colon y(c') \to y(c)} E(tj) \cong Et$.  The right unitor $\rho$ sends ${\blp t,! \colon Et \to S^0 \brp \in SS^0c}$ to $t \in Sc$ and has as its $E$-part the canonical $\colim_{x \colon y(c') \to Et} y(c') \cong Et$.

The left unitality law making $\g$ a colax bifunctor amounts to the following for $(S,E)$ a representation from $\C'$ to $\C$:
\bctikz
1_{\C} \g(S,E) \arrow[equals]{d}[swap]{\wr} & \g(S^0,E^0)\g(S,E) \arrow{l}[swap]{\epsilon \cdot \id} \\
\g(S,E) & \g(S^0S,E^0E) \arrow{l}[swap]{\g(\lambda)} \arrow{u}[swap]{\pi}
\ectikz

By \cref{familial_identitor} and \cref{identitor_composite}, for $\Phi \colon \sint S^0S \to Fib(\C')$ corresponding to $\g(S^0,E^0)\g(S,E)$ and $t \in Sc$, $\epsilon \cdot \id$ is the composite of a cartesian morphism containing $\lambda^S$ and the vertical morphism given by the natural transformation
$$\sint Et \atol{(\id_c,-)} \bigint \blp \C/c \atol{t} \sint S \atol{E} DFib(\C')\brp = \Phi(*_c,t),$$
which composed with the transformation $(\pi_0)_\ast$ contracting to identities the morphisms of the form $\blp i \colon c'' \to c', \id_t \brp$ in $\Phi(*,t)$ yields the canonical isomorphism ${E(t) \cong \colim_{j \colon y(c') \to y(c)} E(tj)}$, inverse to that of $\lambda$. The top composite then amounts to 
\bctikz
& \soint E^0E \arrow{dl}[swap]{} \arrow{r}{} & \int S^0S \arrow{dr}{} \arrow[equals]{d}{} \\
\C' & \soint E \lar \rar \uar[swap]{\int (\lambda^E)^{-1}} \dar[equals] & \sint S^0S \rar \dar[swap]{\int \lambda^S} & \C \\
& \soint E \arrow{ul}{} \arrow{r}[swap]{} & \sint S \arrow{ur}[swap]{}
\ectikz
which by \cref{vertical_iso} is precisely $\g(\lambda)$, so the diagram commutes.

The right unitality law is the square:
\bctikz
\g(S,E) 1_{\C'} \arrow[equals]{d}[swap]{\wr} & \g(S,E)\g(S^0,E^0) \arrow{l}[swap]{\id \cdot \epsilon} \\
\g(S,E) & \g(SS^0,EE^0) \arrow{l}[swap]{\g(\rho)} \arrow{u}[swap]{\phi}
\ectikz

Similarly, for $\Phi \colon \sint SS^0 \to Fib(\C')$ corresponding to $\g(S,E)\g(S^0,E^0)$ and $t \in Sc$, $\id \cdot \epsilon$ is the vertical morphism given by the natural transformation
$$\sint Et \atol{(-,\id)} \bigint \blp\sint Et \to \C' \atol{\C'/-} DFib(\C')\brp = \Phi(t,!),$$
which composed with the transformation $(\pi_0)_\ast$ contracting morphisms of the form 
$$(i_t,\id_{c'}) \colon \blp Et_i(x), ji \colon c'' \to c \brp \to \blp x, j \colon c' \to c \brp$$ 
in $\Phi(t,!)$ for $i \colon c'' \to c'$ yields the canonical isomorphism $E(t) \cong \colim_{x \colon y(c') \to Et} y(c')$, inverse to that of $\rho$. The top composite then agrees with $\int \rho$ as in the left unitality square.

\subsection{Associator and Associativity}

In this section, we fix the following familial polynomials
$$\C''' \atol{r = \g(S'',E'')} \C'' \atol{q = \g(S',E')} \C' \atol{p = \g(S,E)} \C,$$
writing $\A_{pq} \to \B_{pq}$ to denote the opfibration in the polynomial $pq$ corresponding to $\Phi_{pq}$, and likewise for the other composites. We will show that the following diagram commutes in $\Poly^f(\C''',\C)$, which using the shorthand $|pq| = \g(SS',EE')$ suggested by \cref{familial_replacement} expresses the associativity law for the colax bifunctor $\g$:
\bctikz
(pq)r \arrow{r}{\alpha^\Poly} & p(qr) \\
{|pq|r} \arrow{u}{\pi \cdot \id} & {p|qr|} \arrow{u}[swap]{\id \cdot \pi} \\
{|(pq)r|} \arrow{u}{\pi} \arrow{r}[swap]{\g(\alpha^\Rep)} & {|p(qr)|} \arrow{u}[swap]{\pi}
\ectikz

For $t \in Sc$ and $f \colon Et \to S'$, we have
$$\B_{|(pq)r|} c = \B_{|pq|r} c = (SS')S'' c = \coprodl_{(t,f) \in SS'c} \Hom_{\chpp}(\colim(E' \circ \sint f),S'')$$
$$\B_{(pq)r} c = P_d(pq)(S'') c = \coprod_{(t,f)} \Hom_{Fib(\C'')}(\soint (E' \circ \sint f), \sint S'')$$ 
where $\sint Et \atol{\sint f} \sint S' \atol{E'} \chpp$. The $\B$-component of $\pi \colon |(pq)r| \to |pq|r$ is the identity and that of $\pi \cdot \id$ is the map $\B_{|pq|r} \to \B_{(pq)r}$ induced by the maps $Q \colon \soint (E' \circ \sint f) \to \int \colim(E' \circ \sint f)$, which is an isomorphism by \cref{familial_productor} as $\sint S''$ is a discrete fibration. Meanwhile, as $P_d(q)(S'') = S'S''$, we have 
$$\B_{|p(qr)|} c = \B_{p|qr|} c = \B_{p(qr)} c = P_d(p)(S'S'') = S(S'S'')_c = \coprodl_{t \in Sc} \Hom_{\chpp}(Et,S'S'')$$ 
with the $\B$ component of $|p(qr)| \atol{\pi} p|qr| \atol{\id \cdot \pi} p(qr)$ the identity. 

The associator $\alpha^\Rep$ has $S$-component given by 
$$\blp t \in Sc,f \colon Et \to S',F \colon \colim(E' \circ \sint f) \to S''\brp \in (SS')S''c \quad \mapsto \quad (t,G) \in S(S'S'')c$$ 
where $G$ sends $x \in Et_{c'}$ to 
$$\blp f(x) \in S'c', F_x \colon E'f(x) \to \colim(E' \circ \sint f) \atol{F} S''\brp \in S'S''c'.$$ 
This is an isomorphism as every such $G$ uniquely arises in this way: if $G$ sends $x$ to ${\blp g(x) \in S'c', G_x \colon E'g(x) \to S'' \brp}$ we can set $f(x) = g(x)$ and use $(G_x)$ to induce $F$ from the colimit.

The associator $\alpha^\Poly$ has $\B$-component
$$\coprod_{(t,f)} \Hom_{Fib(\C'')} \blp \soint (E' \circ \sint f), \sint S'' \brp \to \coprodl_{t \in Sc} \Hom_{\chpp}(Et,S'S'')$$
sending $(t,f,\bar F)$ in the domain, where by \cref{familial_productor} $\bar F \colon \soint (E' \circ \sint f) \to \sint S''$ must be of the form $Q \circ \sint F$ for some $F \colon \colim(E' \circ \sint f) \to S''$, to the same $G$ constructed for $F$ above. This shows that the $\B$-parts of the associativity diagram commute.

We now fix $(t,f,F)$ as above along with the corresponding $\bar F$ and $G$ that they determine. The fiber of $\A_{(pq)r}$ over $(t,f,\bar F)$ is 
$$\bigint \blp \soint(E' \circ \sint f) \atol{Q} \sint \colim(E' \circ \sint f) \atol{\sint F} \sint S'' \atol{E''} \chppp\brp,$$
while the fiber of $\A_{|(pq)r|}$ over $(t,f,\bar F)$ is 
$$\sint \colim \blp \sint \colim(E' \circ f) \atol{\sint F} \sint S'' \atol{E''} \chppp\brp.$$
The $\A$-part of the map $(\pi \cdot \id) \circ \pi$ between them contracts first the inner lax colimit to a strict colimit, then contracts the outer lax colimit to a strict colimit (which $Q$ no longer affects, see \cite[Example 4.8]{tholencolimit}). The fiber of $\A_{p(qr)}$ over $(t,G)$ is 
$$\bigint \blp x \mapsto \bigint \blp \sint E'f(x) \atol{\sint F_x} \sint S'' \atol{E''} \chppp \brp\brp$$
for $x$ in $Et$, while the fiber of $\A_{|p(qr)|}$ over $(t,G)$ is
$$\sint \colim \blp x \mapsto \colim(E'' \circ \sint F_x)\brp.$$
Similarly, the $\A$-part of the map $(\id \cdot \pi) \circ \pi$ between them contracts first the inner and then the outer lax colimits to strict ones.

The $E$ component of $\alpha^\Rep$ is given by the colimit decomposition isomorphism (see \cite[Lemma 7.13]{BatBerg}, \cite[Theorem 5.4]{tholencolimit}),
$$\colim_{x' \in \colim_{x \in Et} E'f(x)} E''F(x') \cong \colim_{x \in Et} \colim_{x' \in E'f(x)} E''F_x(x'),$$
while the $\A$ component of $\alpha^\Poly$ is the analogous isomorphism for lax colimits, sending $\blp (x,x'),x'' \brp$ to $\blp x,(x',x'') \brp$ as in \cref{opfibration_composite}. Both sides of the $\A$-part of the associativity equation, where the cartesian associator maps act in the direction indicated and the vertical productor maps act in the opposite direction, send $\blp (x,x'),x'' \brp$ to the image of $\blp x,(x',x'') \brp$ in the strict double colimit, either by first reindexing with $\alpha^\Poly$ then quotienting twice from lax to strict colimits or first quotienting then reindexing with $\alpha^\Rep$.

In conclusion, the diagram commutes, showing $\g$ satisfies the associativity law for colax bifunctors. In somewhat counterintuitive fashion, we have now proven that $\g$ is by all accounts a colax bifunctor before proving that the identity, unitor, product, and associator in $\Rep$ satisfy the laws of a bicategory. This is resolved below, but for now we summarize our results on $\g$ in the following:

\begin{thm}\label{is_colax}
$\g \colon \Rep \to \Poly^{vf}$ has the data, structure, and properties of a colax bifunctor.
\end{thm}

\begin{rem}\label{like_bifunctor}
$P_d \colon \Poly^{vf} \to \CAT$ is a bifunctor which sends the colax structure maps $\epsilon,\pi$ of $\g$ to isomorphisms, so the composite $P_d\g$ has the structure and properties of not just a colax bifunctor, but an actual bifunctor, as the structure maps of the composite are build out of those of $P_d$ along with $P_d$ applied to those of $\g$, and all of these maps are isomorphisms. 
\end{rem}

\subsection{Pentagon and Triangle Laws}

To demonstrate that the pentagon and triangle laws hold in $\Rep$, we recall the following more general fact about bicategories: 

\begin{prop}\label{reflects_triangle}
Assume $\bicat{A}$ denotes all of the data and structure of a bicategory (objects, categories of morphisms, identities, composition functors, unitors $\lambda_\bicat{A},\rho_\bicat{A}$, associator $\alpha_\bicat{A}$), $\bicat{B}$ is a bicategory, and $H$ contains the data, structure, and properties of a bifunctor $\bicat{A} \to \bicat{B}$ (mapping on objects, functors between morphism categories, productor $\pi$, identitor $\epsilon$, unitality and associativity equations) such that the functor $H_{\C',\C} \colon \bicat{A}(\C',C) \to \bicat{B}(\C',\C)$ is faithful for all objects $\C,\C'$ in $\bicat{A}$. Then $\bicat{A}$ is a bicategory and $H$ is a bifunctor.
\end{prop}

\begin{proof}
By definition of $\bicat{A}$ and $H$, it only remains to show that the pentagon and triangle laws hold in $\bicat{A}$, and as $H$ is locally injective on 2-cells, it suffices to show that the images of the pentagon and triangle diagrams commute in $\bicat{B}$. We show this for the triangle law; the (much larger) diagram for the pentagon is constructed similarly from the pentagon diagram in $\bicat{B}$. For composable 1-cells $p,q$ in $\bicat{A}$, we have the following diagram:
\bctikz[column sep=small]
H \blp (p1)q \brp \arrow{rrrr}{H(\alpha_\bicat{A})} \arrow{d}[swap]{\pi^{-1}} &&&&  H \blp p(1 q) \brp \arrow{d}{\pi^{-1}} \\
H(p1)H(q) \arrow{d}[swap]{H(\rho_\bicat{A}) \cdot \id} \arrow{r}{\pi^{-1} \cdot \id} & \blp H(p) H(1) \brp H(q) \arrow{rr}{\alpha_\bicat{B}} \arrow{d}[swap]{\id \cdot \epsilon^{-1} \cdot \id} &&  H(p) \blp H(1) H(q) \brp \arrow{d}{\id \cdot \epsilon^{-1} \cdot \id} &  H(p)H(1q) \arrow{l}[swap]{\id \cdot \pi^{-1}} \arrow{d}{\id \cdot H(\lambda_\bicat{A})} \\
H(p)H(q) \arrow{rrdd}[swap]{\pi} & \arrow{l}[swap]{\rho_\bicat{B} \cdot \id} \blp H(p)1 \brp H(q) \arrow{rr}{\alpha_\bicat{B}} \arrow{dr}[swap]{\rho_\bicat{B} \cdot \id} &&  H(p) \blp 1H(q) \brp \arrow{dl}{\id \cdot \lambda_\bicat{B}} \arrow{r}{\id \cdot \lambda_\bicat{B}} & H(p)H(q) \arrow{lldd}{\pi} \\
&&  H(p) H(q) \arrow{d}{\pi} \\
&& H(pq)
\ectikz
The diagram commutes by the associativity and unitality of $H$, naturality of $\alpha_\bicat{B}$, and the triangle law for $\bicat{B}$.  The outer left and right composites equal $H(\rho_\bicat{A} \cdot \id)$ and $H(\id \cdot \lambda_\bicat{A})$, respectively, by naturality of $\pi$, so this is precisely the image under $H$ of the triangle diagram for $p,q$ in $\bicat{A}$.  
\end{proof}

%

\subsection{Monad Representations}\label{rest_of_equations}

Here we give the complete list of equations defining a formal monad in $\Rep$, or equivalently a monoid in the monoidal category $\Rep(\C,\C)$, completing the description in \cref{monad_representation}.

A familial monad representation is given by a familial representation $(S,E)$ from $\C$ to $\C$ equipped with morphisms of representations $e \colon (S^0,E^0) \to (S,E)$ and $m \colon (SS,EE) \to (S,E)$ satisfying the following equations:

\begin{itemize}
\item Left unitality:

\bctikz
S^0S \arrow{dr}[swap]{\lambda^S} \arrow{r}{e^S;id} & SS \arrow{d}{m^S} \\
 & S
\ntikz
\colim_{\iota \colon y(c') \to y(c)} E(t\iota) & \colim_{x \colon y(c') \to Ee^S(*_c)} E(te^E(x)) \arrow{l}[swap]{\colim_{e^E}id} \\
& Et = Em(e^S(*_c),t \colon y(c) \to S) \arrow{ul}{\lambda^E} \arrow{u}[swap]{m^E}
\ectikz

\item Right unitality:

\bctikz
SS \arrow{d}[swap]{m^S} & SS^0 \arrow{dl}{\rho^S} \arrow{l}[swap]{id;e^S} \\
 S
\ntikz
\colim_{x \colon y(c') \to Et} Ee^S(*_{c'}) \arrow{r}{\colim_{id} e^E} & \colim_{x \colon y(c') \to Et} y(c') \\
Em(t, e^S \circ ! \colon Et \to S^0 \to S) = Et \arrow{ur}[swap]{\rho^E} \arrow{u}{m^E}
\ectikz

\item Associativity, where for $t \in Sc$, $f \colon Et \to S$, $F \colon \colim_{x \colon y(c') \to Et} Ef(x) \to S$, $\alpha^S((t,f),F)$ is given by $(t,G)$ where $G \colon Et \to SS$ with $G(x) = (f(x),F|_{Ef(x)})$:

\bctikz
(SS)S \arrow{rr}{\alpha^S} \arrow{d}[swap]{m;id} & & S(SS) \arrow{d}{id;m} \\
SS \arrow{dr}[swap]{m} & & SS \arrow{dl}{m} \\
& S
\ectikz
\bctikz[column sep=small]
\colim_{x' \colon y(c'') \to \colim_{x \colon y(c') \to Et} Ef(x)} EF(x') & \colim_{x \colon y(c') \to Et} \colim_{x' \colon y(c'') \to Ef(x)} EG(x)(x') \arrow{l}[swap]{\alpha^E} \\
\colim_{x' \colon y(c'') \to Em(t,f)} EF(m^Ex') \arrow{u}{\colim_{m^E}id} & \colim_{x \colon y(c') \to Et} Em(f(x),G(x)) \arrow{u}[swap]{\colim_{id} m^E} \\
Em^S(m^S(t,f),Fm^E) \rar[equals] \arrow{u}{m^E} & Em^S(t, m^S(f(x),G(x))) \arrow{u}[swap]{m^E}
\ectikz
\end{itemize}

\bibliography{mybib}
\bibliographystyle{alpha}

\end{document}